\documentclass[11pt, a4paper, reqno, oneside]{amsart}
\usepackage{vmargin,color}
\usepackage[latin1]{inputenc}
\usepackage{amsmath,amsfonts,amsthm,epsfig,graphicx}
\usepackage{enumitem}
\usepackage{caption}
\usepackage{ esint }
\usepackage[none]{hyphenat}
\usepackage{titlesec}
\usepackage{comment}
\titleformat{\section}[block]
  {\centering\Large\bfseries}   
  {\thesection.}                
  {1em}                         
  {}                            

\titleformat{\subsection}[block]
  {\centering\bfseries}
  {\thesubsection.}
  {1em}
  {}

\titlespacing*{\subsection}{0pt}{3.5ex plus 1ex minus .2ex}{1em}

\def\dist{{\rm dist}}

\newtheorem{theorem}{Theorem}[section]
\newtheorem{definition}[theorem]{Definition}

\newtheorem{proposition}[theorem]{Proposition}
\newtheorem{lemma}[theorem]{Lemma}
\newtheorem{corollary}[theorem]{Corollary}

\newtheorem{assumption}[theorem]{Assumption}
\theoremstyle{remark}
\newtheorem{remark}{Remark}[section]

\numberwithin{equation}{section}
\numberwithin{figure}{section}

\address{Department of Mathematics, The University of Texas at Austin, 2515 Speedway, Stop C1200, Austin TX 78712-1202, USA}
\email{antoniofarah@utexas.edu}

\title{A Degenerate One-Phase Free Boundary Problem Arising From the Alt-Phillips Equation for Negative Powers}
\author{Antonio Farah}
\begin{document}

\begin{abstract}
We study viscosity solutions for a class of degenerate one-phase free boundary problems of the form $w\Delta w = h(\nabla w)$. We assume the existence of a star-shaped domain $D$ such that $h < 0$ in $D$, $h = 0$ on $\partial D$, and $h > 0$ in $\bar{D}^{c}$. This class of degenerate one-phase free boundary problems arises when a canonical transformation is performed to a semilinear equation $\Delta u = f(u)$, and $f$ behaves like $-\gamma u^{-(\gamma + 1)}$ for some $\gamma \in (0,2)$. In this case, known as the Alt-Phillips equation for negative power potentials, $h(\rho) = c(|\rho|^2 - 1)$. We show existence of a viscosity solution, Lipschitz regularity, and regularity of the free boundary at flat points. Additionally, we show that as $\gamma$ converges to $2$, the free boundary converges to a minimal surface.  
\end{abstract}
\maketitle

\section{Introduction}
In this work, we study non-negative viscosity solutions to a general class of free boundary problems of the type
\begin{align}\label{eq main problem}
\Delta w &= \frac{h(\nabla w)}{w} &&  \text{ in } B_1^{+} = B_1 \cap \{w > 0\}, \\ \nonumber
\partial^{-}_{\nu \nu} &w = 0 && \text{on }FB(w):= \partial \{w > 0\}. \nonumber
\end{align}
We assume that the right-hand side $h$ of the interior equation 
\begin{equation}\label{eq interior}
\Delta w = \frac{h(\nabla w)}{w}
\end{equation}
is a $C^{1}(\mathbb{R}^n)$ function that vanishes on the boundary of a star-shaped domain $D$, takes negative values inside $D$, and takes positive values in $\bar{D}^{c}$. This class of nonlinearities occurs when studying critical points of the Alt-Phillips functional for negative power potentials, which were recently studied by De Silva and Savin in \cite{de2022alt}. In this particular case,
\begin{equation*}
h(\rho) = c_{\gamma}(|\rho|^2 - 1),
\end{equation*}
$D = B_1$, and $c_{\gamma} > 0$ depends on the particular power $\gamma \in (-2,0)$.

The interesting feature of the problem is the non-standard free boundary condition 
\begin{equation}\label{eq free boundary}
\partial_{\nu\nu}^{-} w = 0,
\end{equation}
which is to be understood as a condition on the degree of the second term in the expansion of the solution $w$ at a free boundary point. Specifically, the condition means that when $w$ is expanded in terms of the distance $d$ to the free boundary, the next term in the expansion after a linear term in $d$ must be $o(d^{1 + \mu})$ for any $\mu \in (0,1)$. This is in contrast to the situation of Alt-Phillips free boundary problems with non-negative power potentials, in which the free boundary condition is imposed on the first term in the expansion. This class of problems includes the classical Alt-Caffarelli problem and the obstacle problem.   

An important observation in the study of these problems is that both the interior equation \eqref{eq interior} and the free boundary condition \eqref{eq free boundary} remain invariant with respect to the Lipschitz rescalings 
\begin{equation}\label{eq rescaling}
w_r(x) = \frac{w(rx)}{r}.
\end{equation}

Before giving the precise assumptions and definitions about the nonlinearity $h$ and the free boundary condition \eqref{eq free boundary}, we provide some context regarding this class of free boundary problems.  




\subsection{Motivation for the problem}

In the theory of semilinear equations, there is a well-known criterion (see \cite{pucci2007maximum}) to determine whether a solution can develop a free boundary or not. Namely, given a non-negative solution $u$ to a semilinear equation $\Delta u = g(u)$ in a bounded domain $\Omega \subset \mathbb{R}^n$ with $g \geq 0$ being an increasing function, we define 
\begin{equation}\label{eq Pucci}
w = \int_{0} ^ {u} \frac{1}{\sqrt{2G(s)}}ds,
\end{equation}
where 
\begin{equation*}
G(s) = \int_{0} ^ {s} g(t) dt.    
\end{equation*}
Then, if $w = \infty$, the problem admits a strong maximum principle, and thus the sets $\{u > 0\}$ and $\{u = 0\}$ cannot both be non-empty. But if $w < \infty$, then a non-trivial free boundary $\partial \{u > 0\}$ may form under appropriate boundary conditions. In this case, $w$ solves the degenerate one-phase problem 
\begin{align}\label{eq semilinear to onephase}
        \Delta w(x) = a(x)\Bigg(\frac{1-|\nabla w(x)|^2}{w(x)}\Bigg) \quad \mbox{in $\{w >0\}$}, \\ \nonumber
        |\nabla w(x)|^2=1 \hspace{3.1cm} \mbox{on $FB(w)$}, \nonumber
\end{align}
where 
\begin{equation*}
a(x) =  \frac{g(u(x)) w(u(x))}{2\sqrt{G(u(x))}}.    
\end{equation*}
The usefulness of \eqref{eq semilinear to onephase} is given by the fact that the equation now enjoys a transmission condition on the free boundary. To illustrate this, we discuss the Alt-Phillips problem for positive power potentials, which serves as the prototypical case for other semilinear equations. 

The Alt-Phillips problem for positive power potentials consists of minimizing the energy functional 

\begin{equation}\label{eq Alt-Phillips positive powers functional}
AP(v) = \int_{\Omega} \frac{|\nabla v|^2}{2} + (v)^{\gamma}_{+},    
\end{equation}
in a domain $\Omega$ among all non-negative $W^{1,2}(\Omega)$ functions that share a fixed boundary datum, and where $\gamma \in (0,2)$. This model was first studied in a series of works by Phillips (\cite{phillips1983hausdoff}, \cite{phillips1983minimization}) and then by Alt and Phillips in \cite{phillips1986free}. An interesting feature of the Alt-Phillips energy functional is that it interpolates between the obstacle problem and the Alt-Caffarelli problem, since
\begin{align}
\int_{\Omega} (v)^{\gamma}_{+} \rightarrow \int_{\Omega} (v)_{+} && \text{ as } \gamma \rightarrow 1 \nonumber \\
\int_{\Omega} (v)_{+}^{\gamma} \rightarrow |v > 0| && \text{ as } \gamma \rightarrow 0. \nonumber 
\end{align}
Critical points then satisfy the equation
\begin{align}\label{eq Alt-Phillips}
\Delta u &= \gamma u^{\gamma - 1} && \text{ in } \Omega^{+} \\
u &= 0 && \text{ in } \Omega \setminus \Omega^{+}. \nonumber
\end{align}
These equations encompass a wide range of semilinear equations, and the right-side of many semilinear equations oscillates between distinct powers of Alt-Phillips (see \cite{allen2025free} and \cite{caffarelli2025grad}). 

The Alt-Phillips equations for positive power potentials mark the first known instance where the transformation \eqref{eq Pucci} was used. In an Alt-Phillips equation for positive power potentials, such transformation corresponds to setting 
\begin{equation}\label{eq Alt-Phillips w}
w = u^{\frac{1}{\beta}},
\end{equation}
where $\beta = \frac{2}{2-\gamma}$. Then, a critical point $w$ now solves
\begin{align}\label{eq w for Alt Phillips}
&\Delta w = \frac{h(\nabla w)}{w} && \text{ in }  \Omega^{+},  \\ \nonumber 
&|\nabla w|^2 = 1 && \text{ on } FB(w), \nonumber 
\end{align}
where 
\begin{equation*}
h(\rho) = \frac{\gamma}{1 - \gamma}(1 - |\rho|^2).
\end{equation*}
The usefulness of the transformation can best be seen in the works of Phillips \cite{phillips1983hausdoff} and \cite{phillips1986free}. In such work, the authors use the fact that $|\nabla w|^2 \rightarrow 1$ as the free boundary is approached to derive measure-theoretic estimates for the free boundary; they do so by establishing an asymptotic rate of convergence, which enables them to deduce integrability bounds for $w^{-1}$ near the free boundary. Additionally, we remark that the free boundary condition $|\nabla w|^2 = 1$ does not apply to all solutions of the Alt-Phillips equation \eqref{eq Alt-Phillips}, but rather applies to critical points of the energy functional \eqref{eq Alt-Phillips positive powers functional}.

Motivated by \eqref{eq w for Alt Phillips}, De Silva and Savin in \cite{de2021certain} studied properties of solutions to degenerate one-phase problems of the form 
\begin{align}\label{eq De Silva Savin}
&\Delta w = \frac{h(\nabla w)}{w} && \text{ in } \Omega^{+} \\\nonumber 
&\nabla w \in \partial D && \text{ on } FB(w), \\ \nonumber 
\end{align}
with the nonlinearity $h$ satisfying the sign assumptions
\begin{equation*}
h > 0 \qquad \text{ in } D, \qquad h = 0 \qquad \text{ on } \partial D, \qquad h < 0 \qquad \text{ in } \bar{D}^{c},
\end{equation*}
\noindent
where $D$ is a star-shaped domain.
We note that, in such work, the authors study viscosity solutions to \eqref{eq De Silva Savin}, which do not necessarily arise as a transformation of a semilinear equation. Strictly using non-variational techniques, they show an optimal Lipschitz regularity bound, measure theoretic estimates for the free boundary, and an improvement of flatness argument near regular free boundary points. Importantly, the free boundary condition $\nabla w \in \partial D$, which is to be understood in the viscosity sense as well, is both a necessary and sufficient condition needed to obtain all the results detailed in the work.

Recently, the Alt-Phillips equations for negative power potentials have attracted some attention in the literature, as they are a natural extension of the Alt-Phillips model for non-negative power potentials. Specifically, the former equations arise from studying critical points of the energy functional 

\begin{equation}\label{eq Alt-Phillips negative powers functional}
AP(v) = \int_{\Omega} \frac{|\nabla v|^2}{2} + (v)^{-\gamma}_{+} \qquad \gamma \in (0,2)   
\end{equation}
\noindent
among all non-negative $W^{1,2}(\Omega)$ functions that satisfy the same boundary datum. Critical points then solve the interior equation
\begin{equation}\label{eq Alt-Phillips Negative Powers}
\Delta u = -\gamma u^{-(\gamma + 1)} \qquad \text{ in } \Omega^{+},
\end{equation}
along with the free boundary condition
\begin{equation}\label{eq free boundary negative exponents}
u(x_0 + t\nu) = c_0t^{\alpha} + o(t^{2-\alpha}) \qquad \text{ on $FB(u)$}
\end{equation}
\noindent
for $\alpha = \frac{2}{2 + \gamma}$. In particular, \eqref{eq free boundary negative exponents} means that given a free boundary point $x_0 \in FB(u)$ and a ball $B$ such that $x_0 \in \partial B$, $u$ cannot be touched from below (resp. above) by a function of the form

\begin{equation}\label{eq free boundary negative exponents explanation}
c_0d(x)^{\alpha} \pm \epsilon d(x)^{2 - \alpha},  
\end{equation}
for any $\epsilon > 0$, where
\begin{equation*}
c_0 = [\alpha(1-\alpha)]^{-\frac{1}{\gamma + 2}},
\end{equation*}
and 
\begin{align*}
d(x) = 
\begin{cases}
dist(x, \partial B) \qquad & \text{ for } x \in B \\
0 \qquad & \text{ otherwise}.
\end{cases}
\end{align*}
\noindent
In fact, \eqref{eq free boundary negative exponents} is a predecessor to our free boundary condition \eqref{eq free boundary}. 

These problems were studied, for example, by De Silva and Savin in \cite{de2022alt}, \cite{de2022uniform}, and \cite{de2023compactness}. Aside from the optimal Lipschitz regularity and improvement of flatness bounds, a very interesting feature of the problems is that they have a connection to the minimal surfaces equation as $\gamma \rightarrow 2$. In \cite{de2022alt}, it was shown that the energy functionals \eqref{eq Alt-Phillips Negative Powers} $\Gamma$-converge to $\int_{\Omega}Per\{u > 0\}$ as $\gamma \rightarrow 2$; the connection was further explored in \cite{de2022uniform} and \cite{de2023compactness}. 

When the transformation \eqref{eq semilinear to onephase} is applied to a minimizer of the Alt-Phillips equation for negative power potentials, the transformed function satisfies the free boundary problem
\begin{align}\label{eq w for Alt Phillips negative}
&\Delta w = \frac{h(\nabla w)}{w} && \text{ in } \Omega \cap \{w > 0\} := \Omega^{+} \nonumber \\
&\partial_{\nu\nu}^{-} w  = 0 && \text{ on } \partial \{w > 0\} \cap \Omega, \nonumber \\
\end{align}
\noindent
where 
\begin{equation*}
h(\rho) = \frac{\gamma}{2 + \gamma}(1 - |\rho|^2).
\end{equation*}
\noindent
Our free boundary problem \eqref{eq main problem} can then be seen as a generalization of \eqref{eq w for Alt Phillips negative}. We now precise the exact nature of the right-hand side $h$ and the free boundary condition \eqref{eq free boundary} for the free boundary problem \eqref{eq main problem}. 
\begin{assumption}[Standard Assumptions of the Problem]\label{assumption standard assumption}
Let $D$ be a bounded star-shaped domain defined by a $C^2$ function $f: \mathbb{S}^{n-1} \rightarrow \mathbb{R}$, referred to as the domain function, given by the property that
\begin{equation}\label{eq f definition}
f(\nu) \nu \in \partial D
\end{equation}
for all $\nu \in \mathbb{S}^{n-1}$. Then, we stipulate that the nonlinearity $h \in C^{1}(\mathbb{R}^n)$ be such that
\begin{equation*}
h < 0 \qquad \text{ in } D, \qquad h = 0 \qquad \text{ on } \partial D, \qquad h > 0 \qquad \text{ in } \bar{D}^{c}.
\end{equation*}
Furthermore, we assume the following bound on the radial derivative of $h$ on $\partial D$
\begin{equation}\label{eq mu condition}
\nabla h(f(\nu)\nu) \cdot \nu < f(\nu)    
\end{equation}
for all $\nu \in \mathbb{S}^{n-1}$. 
\end{assumption}

\begin{definition}[Free Boundary Condition]\label{def free boundary}
Let $x_0 \in \partial \{w > 0\} \cap \Omega$, and let us suppose that $\Gamma$ is an oriented $C^3$ hypersurface with normal $\nu$ such that $x_0 \in \Gamma$. Then, we let $d(x)$ be the distance to $\Gamma$ on the side of $\Gamma$ to which $\nu$ is interior, 
so that $\nu(x) = \nabla d(x)$. We say that $w$ satisfies the supersolution free boundary condition $\partial_{\nu \nu}^{-}w(x_0) \leq 0$ at $x_0$ if $w$ cannot be touched from below by a function of the form
\begin{equation}\label{eq free boundary condition supersol}
u(x) = f(\nu(x))d(x) + \epsilon d(x)^{1 + \mu},    
\end{equation}
for any  $\epsilon > 0$ and $\mu \in (0, 1)$, where $f$ is defined in \eqref{eq f definition} in Assumption \ref{assumption standard assumption}. 

Similarly, we say that $w$ satisfies the subsolution free boundary condition $\partial_{\nu \nu}^{-}w(x_0) \geq 0$ at $x_0$ if $w$ cannot be touched by above by a function of the form
\begin{equation}\label{eq free boundary condition subsol}
u(x) = f(\nu(x))d(x) - \epsilon d(x)^{1 + \mu}.   
\end{equation}
We say that $\partial_{\nu \nu}^{-} w(x_0) = 0$ if both $\partial_{\nu \nu}^{-}w(x_0) \leq 0$ and $\partial_{\nu \nu}^{-}w(x_0) \geq 0$.

Finally, we say that
\begin{equation*}
\partial_{\nu \nu}^{-} w = 0 \qquad \text{in }FB(w)
\end{equation*}
if $\partial_{\nu \nu}^{-}w(x) = 0$ for all $x \in FB(w)$.
\end{definition}
Hence, the $w$ corresponding to the Alt-Phillips equation for negative exponents, defined in \eqref{eq w for Alt Phillips negative}, solves a free boundary problem of the form \eqref{eq main problem}, with $D$ being the unit sphere. This work can be seen as the analog of \cite{de2021certain} for an Alt-Phillips equation with negative power potentials. Additionally, the connection with minimal surfaces generalizes to this case as well, which will be studied in the last section of this work.

However, there are important differences between our free boundary problem \eqref{eq main problem} and the generalization of the Alt-Phillips equation for positive power potentials \eqref{eq De Silva Savin} studied in \cite{de2021certain}. The latter work considers an equation that is superharmonic when the gradient is large. This condition by itself already implies that the solution should enjoy some regularity; solutions to equations of this form were shown to have universal integrability and regularity properties by Imbert and Silvestre in \cite{imbert2016estimates}. However, an equation that is subharmonic when the gradient is large does not necessarily necessarily satisfy such estimates; as a result, stronger conditions on the right-hand side must be assumed to show regularity estimates. For example, while a  family of solutions to the interior equation appearing in the free boundary problem \eqref{eq De Silva Savin} satisfies a Harnack inequality when we assume the existence of a constant $C_h > 0$ such that $h(\rho) \geq -C_h|\rho|^2$, a family of solutions to the interior equation \eqref{eq interior} will only satisfy a Harnack inequality when $h(\rho) \leq C_h|\rho|^2$ and $C_h \in (0,1)$. A counterexample when $C_h = 1$ is given in Section 2. Additionally, the interior equation \eqref{eq interior} has the property that a strictly positive solution can always be found for it, even in the setting where the boundary data is zero. This is in contrast to solutions to the interior equation appearing in the free boundary problem \eqref{eq De Silva Savin}; a strictly positive solution to such interior equation cannot be guaranteed for all non-negative boundary data. A simple observation that suggests as much is the existence of the so-called ``bump" subsolution $c(1 - |x|^2)$ for a small constant $c > 0$ to the interior equation \eqref{eq interior}, which is strictly positive in $B_1$ and zero on $\partial B_1$. In conclusion, the regularity and existence treatment for \eqref{eq main problem} is delicate and heavily relies on the assumptions placed on the right-hand side $h$ and on the free boundary condition \eqref{eq free boundary}.

\subsection{The Free Boundary Condition}\label{sub free boundary} 
\vspace{1 em}
We provide the motivation for the free boundary condition \eqref{eq free boundary}. It is inspired by the free boundary condition for critical points of the Alt-Phillips functional for negative power potentials \eqref{eq Alt-Phillips negative powers functional}.


Turning to the one dimensional setting, let us suppose we wish to solve:
\begin{equation*}
w'' = \frac{h(w')}{w},
\end{equation*}
where $h \in C^{1}$ satisfies $h < 0$ on $(-1,1)$, $h = 0$ on $\{-1, 1\}$, and $h > 0$ elsewhere. Solutions to this ODE have the implicit form
\begin{equation*}
w = \int_{C_1} ^ {w'(t)} \frac{h(s)}{s} + C_2
\end{equation*}
for some constants $C_1, C_2$. Then, we can use the fact that
\begin{equation*}
h(s) \sim h'(1)(s - 1)
\end{equation*}
for $s$ sufficiently close to $1$. Thus, near the free boundary, solutions satisfy the ODE

\begin{equation*}
w' = 1 + C_3w^{\mu},
\end{equation*}
where $\mu = h'(1)$. Solving this then yields

\begin{equation*}
w(x) = x \pm s x^{1 + \mu} + o(x^{1 + \mu}).
\end{equation*}
Consequently, if we stipulate that $h'(1) < 1$, $x^{1 + \mu}$ is not twice differentiable at the origin. We can then impose the free boundary condition that $s = 0$, which is a rudimentary version of the free boundary condition \eqref{eq free boundary}. 

With the one-dimensional example as motivation, we now turn to multiple dimensions. Let $\Gamma$ be a smooth surface enclosing a region, let $d(x) = \dist(x, \Gamma)$ on one side of $\Gamma$ and zero elsewhere, and let $\nu(x) = \nabla d(x)$. We then set, as in the one dimensional case,  
\begin{equation}\label{eq sub and super solutions}
u(x) = f(\nu(x))d(x) + sd(x)^{1 + \mu},
\end{equation}
and we look for a condition on $\mu$ so that $u$ is a subsolution when $s > 0$ and a supersolution when $s < 0$ in a neighborhood of $\Gamma$. It then follows that 
\begin{align*}
&\nabla u = f(\nu(x))\nu(x) + s(1 + \mu)d(x)^{\mu}\nu(x) + o\big(d(x)\big) \\
&\Delta u = \kappa(x) f(\nu(x)) + s\mu(1 + \mu)d(x)^{\mu - 1} + o\big(d(x)^{\mu}\big) \\
&u \Delta u = sf(\nu(x)) \mu(1 + \mu)d(x)^{\mu} + \kappa(x)f^2(\nu(x))d(x) + o\big(d(x)^{\mu}\big) \\
&h(\nabla u) = \big(\nabla h(f(\nu(x)) \nu(x))\cdot \nu(x)\big)s(1 + \mu)d(x)^{\mu} + o\big(d(x)^{\mu}\big),
\end{align*}
where $\kappa(x)$ denotes the mean curvature of $\Gamma$ at $x$. Hence, if we set
\begin{equation}\label{eq mu full condition}
1 > \mu > \sup_{|\nu| = 1}\frac{\nabla h\big(f(\nu)\nu\big) \cdot \nu}{f(\nu)},     
\end{equation}
then it will be the case that $u$ is a subsolution when $s > 0$ and a supersolution when $s < 0$ for small values of $d(x)$. We summarize this in the following lemma.
\begin{lemma}\label{Lemma Test Functions}
Let $\Gamma$ be a $C^{3}$ hypersurface. Set $d(x)$ to be the distance to $\Gamma$ on one side of $\Gamma$ and zero elsewhere, and let $\nu(x) = \nabla d(x)$. For $s > 0$ and $\mu$ as in \eqref{eq mu full condition}, the test functions
\begin{equation*}
u(x) = f(\nu(x))d(x) \pm sd(x)^{1 + \mu}
\end{equation*}
are subsolutions ($+$) or supersolutions ($-$) for the interior equation \eqref{eq interior} as long as 
\begin{equation*}
d(x) \leq c(s, \mu, \max \kappa),   
\end{equation*}
where $\max \kappa$ denotes the highest principal curvature of $\Gamma$.
\end{lemma}

It is precisely for this reason that we must then impose \eqref{eq mu condition}, and we may then choose a $\mu$ that will satisfy \eqref{eq mu full condition}. This parallels the requirement that $h'(1) < 1$ in one dimension, and the free boundary condition is analogous to setting $s = 0$ in the one-dimensional expansion. 

An interesting case occurs when we have a sequence of solutions $\{w_k\}$ and associated right hand sides and domain functions $\{h_k\}, \{f_k\}$ satisfying Assumption \ref{assumption standard assumption} in which  
\begin{equation}\label{eq mu limiting case}
\frac{\nabla h_k(f_k(\nu)\nu) \cdot \nu}{f_k(\nu)} \rightarrow 1 
\end{equation}
for all $|\nu| = 1$. If $w_k \rightarrow w, h_k \rightarrow h, f_k \rightarrow f$, and $FB(w_k) \rightarrow FB(w)$ in an appropriate sense, then $w$ will solve an interior equation of the form \eqref{eq interior}. When looking at subsolutions and supersolutions of the form $u$ in \eqref{eq sub and super solutions}, the sign term
\begin{equation}\label{eq curvature term}
\kappa(x)f(\nu(x))d(x)
\end{equation}
now becomes important as the free boundary is approached. For example, if $\kappa(x) > 0$, then $u$ will be a subsolution near the point $x$ regardless of the value of $c_2$. Analogously, if $\kappa(x) < 0$, then $u$ will be a supersolution near $x$ regardless of the value of $c_2$. This then strongly suggests that the free boundary of $w$ must be a minimal surface in the viscosity sense. This result is motivated by the fact that \eqref{eq mu limiting case} represents sending $\gamma \rightarrow 2$ in the Alt-Phillips functional for negative power potentials \eqref{eq Alt-Phillips positive powers functional}. Hence, the free boundary of $w$ being a minimal surface in the viscosity sense is analogous to the $\Gamma$-convergence of the energy functionals to the Perimeter functional, as shown in \cite{de2022alt}.

\subsection{Main Results}
\vspace{1 em}

\textbf{Notational Conventions:} Throughout the work, $c$ and $C$ denote small and large constants, respectively, that depend only on the nonlinearity $h$ and the domain function $f$. Additional dependencies are listed in parentheses next to these constants. 

First, we give an existence and regularity result for strictly positive solutions to the interior equation in the entire unit ball

\begin{equation*}
\Delta w  = \frac{h(\nabla w)}{w} \qquad \text{ in } B_1. 
\end{equation*}

\begin{theorem}[Existence and regularity for the interior equation]\label{thm interior equation}
Let $w \in C(B_1)$ be a positive viscosity solution to the interior equation \eqref{eq interior} in $B_1$. Then, $w \in C^{2, \alpha}(B_1)$ for all $\alpha < 1$, along with the estimate 
\begin{equation*}
||w||_{C^{2,\alpha}(B_{\frac{1}{2}})} \leq C(\alpha)||w||_{L^{\infty}(B_1)}.
\end{equation*}
Moreover, 
\begin{equation*}
c \leq ||w||_{L^{\infty}(B_\frac{1}{2})}.  
\end{equation*}
Additionally, for any non-negative $\phi \in C(\partial B_1)$, there exists a solution to the interior equation \eqref{eq interior} that agrees with $\phi$ on $\partial B_1$. Further,

\begin{equation*}
||w||_{L^{\infty}(B_1)} \leq C(1 + ||\phi||_{L^{\infty}(\partial B_1)}).
\end{equation*}

\end{theorem}
Further, we will give precise conditions on $h$ under which uniqueness can be guaranteed. Next, we discuss properties of solutions to the free boundary problem \eqref{eq main problem}. Our first result is the Lipschitz regularity of solutions, which is optimal. 
\begin{theorem}[Optimal regularity of the free boundary problem]\label{thm Lipschitz Regularity}
Let $w$ be a solution to \eqref{eq main problem} such that $0 \in FB(w)$. Then,
\begin{equation*}
||\nabla w||_{L^{\infty}(B_{\frac{1}{2}})} \leq C,
\end{equation*}
and $w \in C^{2,\alpha}(B_{1}^{+})$ for any $\alpha < 1$.
\end{theorem}
We now show the existence of a Perron solution to the Dirichlet problem.
\begin{theorem}[Existence of the Perron solution]\label{thm Perron Existence}
For any non-negative $\phi \in C(\partial B_1)$, consider the minimizing class 
\begin{equation*}
\mathcal{A} = \{v \in C(\overline{B_1}): v \text{ is a supersolution to \eqref{eq main problem} and } v = \phi \text{ on } \partial B_1\}, 
\end{equation*}
and let $w$ be the minimal Perron solution 
\begin{equation*}
w(x) = \inf_{v \in A} v(x).
\end{equation*}
Then, $w$ is a solution to \eqref{eq main problem} that agrees with $\phi$ on $\partial B_1$. Furthermore, $w$ satisfies the nondegeneracy estimate 
\begin{equation*}
\sup_{x \in B_r(x_0)} w(x) \geq cr \qquad \forall x_0 \in FB(w).
\end{equation*} 
\end{theorem}
Next, we provide a result about the regularity of the free boundary at flat points. This then implies by standard arguments that the free boundary is a $C^{1,\alpha}$ graph around flat points for some $\alpha > 0$ (see, for example, chapter 8 in \cite{velichkov2023regularity}). 

\begin{theorem}\label{thm flatness improv}
Assume that $w$ is a solution to \eqref{eq main problem} in $B_1$ with $0 \in FB(w)$. Then, there exist universal constants $r, \epsilon_0$ with $r \in (0,1)$ such that if for some $\epsilon  \in (0, \epsilon_0)$,
\begin{align*}
f(e_n)x_n - \epsilon \leq w \leq (f(e_n)x_n + \epsilon)_{+} && \text{ in } B_1,
\end{align*}
then,
\begin{align*}
f(\nu) (x \cdot \nu) - \frac{r}{2}\epsilon \leq w \leq \big(f(\nu)(x \cdot \nu) + \frac{r}{2}\epsilon)_{+} && \text{ in } B_r  
\end{align*}
for some $\nu \in \mathbb{S}^{n-1}$. 
\end{theorem}
\noindent
Finally, we highlight the connection with minimal surfaces.

\begin{theorem}\label{thm minimal surfaces}
Suppose we are given a sequence of solutions $\{w_k\}$ to \eqref{eq main problem} along with a sequence of right-hand sides $\{h_k\}$ and domain functions $\{f_k\}$ such that $w_k$ satisfies Assumption \ref{assumption standard assumption} and the free boundary condition \eqref{eq free boundary}. 

Furthermore, suppose that
\begin{equation}
\frac{\nabla h_k (f_k(\nu)\nu) \cdot \nu}{f_k(\nu)} \rightarrow 1
\end{equation}
\noindent
for all $\nu \in \mathbb{S}^{n-1}$ and that
\begin{align}\label{eq convergence}
&w_k \rightarrow w \qquad \text{ uniformly } 
&h_k \rightarrow h \qquad \text{ in } C^1 \\  
&\partial \{w_k >0\} \rightarrow\partial\{w > 0\} \qquad \text{ in Hausdorff distance}. \nonumber
\end{align}
Then, $w$ is a viscosity solution to the interior equation
\begin{equation}\label{eq min surf h}
\Delta w = \frac{h(\nabla w)}{w} \qquad \text{ in } B_{1}^{+},
\end{equation}
and $FB(w)$ is a minimal surface in the viscosity sense.
\end{theorem}
We organize the work as follows. In Section \ref{sec Existence and Regularity}, we study the existence and regularity of the problem in two parts: Subsection \ref{subsec  interior} is dedicated to studying strictly positive solutions to the interior equation \eqref{eq interior} in $B_1$, while Subsection \ref{subsec free boundary} is dedicated to studying solutions to the free boundary problem \eqref{eq main problem}. In Section \ref{sec flatness improv}, we show the improvement of flatness argument detailed in Theorem \ref{thm flatness improv}. Finally, Section \ref{sec minimal surface} is dedicated to proving the connection with minimal surfaces detailed in Theorem \ref{thm minimal surfaces}.

\textbf{Acknowledgments:} The author would like to thank Ovidiu Savin for his continuous advice as well as for his guidance and mentorship throughout this project.

\section{Existence and Regularity}\label{sec Existence and Regularity}
In this section, we prove existence and regularity results for solutions of both the interior equation \eqref{eq interior} and of the free boundary problem \eqref{eq main problem}. In order to do this, we must assume the existence of a constant $C_h \in (0,1)$ such that
\begin{equation}\label{eq assumption c_h}
h(\rho) \leq C_h|\rho|^2,    
\end{equation}
\noindent
for all $\rho \in \mathbb{R}^n$. In the case where $h$ arises from an Alt-Phillips equation for negative power potentials, such assumption is indeed met. 

We first present results for solutions of the interior equation \eqref{eq interior} in $B_1$, and then we use those results for the analysis of solutions of the free boundary problem \eqref{eq main problem}.

\subsection{The Interior Equation} \label{subsec  interior}
We first start by giving an a priori Harnack inequality that establishes the regularity of solutions to the interior equation \eqref{eq interior} in $B_1$. The resulting regularity estimate will then be used to show the existence of a solution to the Dirichlet problem for the interior equation \eqref{eq interior} in $B_1$ for any continuous, non-negative boundary data. 

\begin{lemma}\label{Lemma Harnack Inequality}
Suppose that $w$ is a solution to \eqref{eq interior} in $B_1$, and $\sigma \geq 0$ is such that $w \geq \sigma$ on $B_1$. Then, it follows that 

\begin{equation}\label{eq Harnack}
\sup_{x \in B_{\frac{1}{2}}} [w(x) - \sigma] \leq C[\inf_{x \in B_{\frac{1}{2}}} [w(x) - \sigma]+ 1].
\end{equation}
\end{lemma}

\begin{proof}
For $p > 0$, we compute that
\begin{equation*}
\Delta (w - \sigma)^{p} = p(p-1)(w - \sigma)^{p-2}|\nabla w|^{2} + p\frac{(w - \sigma)^{p-1}}{w}h(\nabla w) 
\end{equation*}
\begin{equation*}
\leq p(w - \sigma)^{p-2}|\nabla w|^2\big((p-1) + C_h \big) \leq 0 
\end{equation*}
if $p$ is sufficiently small. This then implies that 

\begin{equation}\label{eq weak Harnack}
||(w - \sigma)||_{L^p(B_{\frac{3}{4}})} \leq C[\inf_{x \in B_{\frac{1}{2}}}w(x) - \sigma]. 
\end{equation}
Now, we show that $w \geq c$ in $B_{\frac{1}{2}}$. This will then allow us to show that $\Delta w \geq -C$, which will conclude the argument. To see this, consider the family of functions 
\begin{equation*}
\phi_{t} = [t(\frac{1}{2} - |x|^2)]_{+}.
\end{equation*}
As $t \rightarrow 0$, notice that
\begin{equation*}
\Delta \phi_{t} \rightarrow 0, |\nabla \phi_{t}| \rightarrow 0, \phi_{t} \rightarrow 0.
\end{equation*}
Since $h(0) < 0$, $\phi_{t}$ will be a subsolution to \eqref{eq interior} for $t \leq c$ in $B_{\frac{1}{2}}$. Then, if $w < c$ at some point $x_1 \in B_{\frac{1}{2}}$, $\phi_{t}$ will touch $w$ by below for some small $t > 0$ at some $x_2 \in B_\frac{1}{2}$, a contradiction. Hence, $w \geq c$, and so
\begin{equation*}
\Delta (w - \sigma) \geq -C \qquad \text{ in } B_{\frac{1}{2}}.
\end{equation*}
It then follows that  

\begin{equation*}
\Delta \big((w - \sigma) + C|x|^{2}\big) \geq 0. \qquad
\end{equation*}
\noindent
Thus, 

\begin{equation}\label{eq weak maximum principle}
\sup_{x \in B_{\frac{1}{2}}} w(x)-\sigma \leq C||(w - \sigma) + C|x|^{2}||_{L^{p}(B_{\frac{3}{4}})} \leq C[||(w - \sigma)||_{L^{p}(B_{\frac{3}{4}})} + 1].   
\end{equation}
\noindent
Combining \eqref{eq weak Harnack} with \eqref{eq weak maximum principle}, we obtain the result. 
\end{proof}

\begin{remark}
The condition that $C_h < 1$ is somewhat sharp if one wishes to obtain a Harnack inequality. For example, for $j \in \mathbb{N}$ and $t \in \mathbb{R}$, the functions 
\begin{equation*}
g_j(t) = e^{jt}     
\end{equation*}
satisfy the ordinary differential equation 

\begin{equation*}
g_jg_j'' = (g_j')^2.   
\end{equation*}
However, 
\begin{equation*}
g_j(0) = 1 \text{  while  } g_j(1) = e^{j}    
\end{equation*}
for all $j \in \mathbb{N}$. Hence, a family of solutions in the multidimensional case that exhibits this behavior along any direction will serve as a counterexample.

There is another condition that can be imposed that also guarantees the existence of a Harnack inequality. Namely, if there exists some positive constant $\delta$ such that 
\begin{equation*}
 h(\rho) \geq (1 + \delta)|\rho|^2   
\end{equation*}
for all $|\rho| \geq C$, then it can be shown that $v^{-\epsilon}$ will be superharmonic for sufficiently small $\epsilon > 0$. This then implies the following: letting $N = \sup_{B_{\frac{1}{2}}} w$, there exists a constant $C(N)$  with $C(N) \rightarrow 0$ as $N \rightarrow \infty$ such that
\begin{equation}\label{eq Harnack 1}
|B_\frac{1}{2} \cap \{w \leq 1\}|\leq C(N).
\end{equation}
We may then use the ideas of Imbert and Silvestre in Lemma 3.1 of \cite{imbert2016estimates} to show that if $w(0) = 1$, then
\begin{equation}\label{eq Harnack 2}
|B_\frac{1}{2} \cap \{w \leq 1\}| \geq c.
\end{equation}
It then follows that taking $N$ large enough will provide a contradiction.
\end{remark}
We may now show from standard arguments that a solution $w \in C(B_1)$ is in $C^{\alpha}(B_{\frac{1}{2}})$ for some $\alpha < c$. For self-containment, we state the following lemmas, and refer the interested reader to \cite{de2021certain} for the proofs, which are identical for this case. First, we present an oscillation decay estimate, which follows from iterating Lemma \ref{Lemma Harnack Inequality}. 
\begin{lemma}[Oscillation Decay]\label{Lemma Oscillation Decay}
Let $w$ be a solution to the interior equation \eqref{eq interior}. Then, let 
\begin{equation*}
\omega(r) = \sup_{x \in B_r} w(x) -\inf_{x \in B_r} w(x)
\end{equation*}
denote the oscillation of $w$ in $B_r$. Then, there exists some $\gamma \in (0,1)$ such that for any $r \leq \frac{1}{2}$
\begin{equation*}
\omega(\frac{r}{2}) \leq \gamma \omega(r).
\end{equation*}
\end{lemma}
Lemma \ref{Lemma Harnack Inequality} and Lemma \ref{Lemma Oscillation Decay} then imply the Holder continuity estimate.
\begin{lemma} [Holder Continuity]
If a positive function $v$ defined on $B_1$ satisfies the Harnack inequality Lemma \ref{Lemma Harnack Inequality} and the oscillation decay Lemma \ref{Lemma Oscillation Decay}, then there exists some $\alpha < c$ such that $v$ is in $C^{\alpha}(B_{\frac{1}{2}})$, with the estimate
\begin{equation*}
||v||_{C^{\alpha}(B_{\frac{1}{2}})} \leq C(1 + v(0)).
\end{equation*}
\end{lemma}
Regarding the optimal regularity of the solution, we recall Lemma 3.8 in \cite{de2021certain}.
\begin{lemma}[De Silva-Savin]\label{Lemma De Silva Savin}
There exists a fixed constant $\eta > 0$ such that if $w$ solves the inequality
\begin{equation*}
|\Delta w| \leq \eta \qquad \text{ when } |\nabla w| \leq \frac{1}{\eta}
\end{equation*}
in $B_1$, and $||w||_{L^{\infty}(B_1)} \leq 1$, then $w \in C^{1,\alpha}(B_\frac{1}{2})$ for some $\alpha \in (0,1)$ depending on $\eta$.
\end{lemma}
Now, let us suppose that $w$ solves the interior equation \eqref{eq interior} in $B_1$. Then, it follows from the proof of Lemma \ref{Lemma Harnack Inequality} that $w(0) > c$. Consequently, if $\alpha \in (0,1)$ is such that $w \in C^{\alpha}(B_{\frac{1}{2}})$, we have that, for $\kappa = ||w||_{C^{\alpha}(B_{\frac{1}{2}})}$, the function

\begin{equation*}
\tilde{w} = \frac{w(\rho x) - w(0)}{\kappa \rho^{\alpha}},
\end{equation*}
satisfies the assumptions of Lemma \ref{Lemma De Silva Savin} for $\rho < c$. Indeed, the Holder continuity of $w$ implies that 
\begin{equation*}
\tilde{w} \leq 1.
\end{equation*}
Additionally, invoking assumption \eqref{eq assumption c_h} yields that
\begin{equation*}
\Delta \tilde{w}(x) = \frac{\rho^{2 - \alpha}h(\nabla w(\rho x))}{C} \leq \frac{\rho^{\alpha}|\nabla \tilde{w}(x)|^2}{C}, 
\end{equation*}
which implies that
\begin{equation*}
|\Delta \tilde{w}(x)| \leq C \quad \text{ when } \quad |\nabla \tilde w(x)| \leq c.
\end{equation*}
Hence, we arrive to the following corollary.

\begin{corollary}\label{Corollary A priori}
A solution $w$ to \eqref{eq interior} in $B_1$ is in $C^{1, \alpha}(B_\frac{1}{2})$ for some $\alpha < c$. In particular, $|\nabla w| \leq Cw(0)$ on $B_{\frac{1}{2}}$.
\end{corollary}
Since it is known that $w \geq c$ on $B_{\frac{1}{2}}$, Schauder estimates now imply that a solution $w$ to $\eqref{eq interior}$ in $B_1$ is in $C^{2,\alpha}(B_{\frac{1}{2}})$ for any $\alpha \in (0,1)$, along with the estimate 
\begin{equation}\label{eq interior regularity}
||w||_{C^{2,\alpha}(B_{\frac{1}{2}})} \leq C||w||_{L^{\infty}(B_1)}.
\end{equation}
With this regularity estimate in mind, we may show that there exists a solution to \eqref{eq interior} in $B_1$ that agrees with any non-negative boundary datum $\phi \in C(\partial B_1)$.  Even though the proof is similar in spirit and is inspired by the one done for the degenerate one-phase problem arising from the Alt-Phillips equation for positive power potentials studied in \cite{de2021certain}, the sign change introduces significant technical complications. Hence, the proof is first done in the case where $\phi$ is strictly positive; the general case is then handled by approximation. Before the proof, we construct families of subsolutions and supersolutions that satisfy the interior equation \eqref{eq interior} in their positivity set and match any non-negative Holder boundary datum. These will be used, in particular, to show that the approximation argument in the proof of existence includes the convergence of the approximate solutions at the boundary.

\begin{lemma}\label{lemma sup and sub interior}
Given any positive $\phi \in C^{\alpha}(\partial B_1)$, there exists a continuous family of supersolutions and subsolutions of \eqref{eq interior} in $B_1$ that agree with $\phi$ on $\partial B_1$.
\end{lemma}
\begin{proof}
We first define a family of supersolutions. For a positive constant $C_{\mathrm{sup}}$, define
\begin{equation*}
\psi_{C_{\mathrm{sup}}, \phi}(x) = \inf_{x_0 \in \partial B_1} \phi(x_0) + C_{\mathrm{sup}}|(x - x_0) \cdot \nu_{x_0}|^{\frac{\alpha}{2}},
\end{equation*}
where $\nu_{x_0}$ denotes the inner pointing normal to $B_1$ at $x_0$. Now, we check that $\psi_{C_{\mathrm{sup}}, \phi}$ is a supersolution to the interior equation \eqref{eq interior} for $C_{\mathrm{sup}} \geq C$. To do so, it suffices to check that the function
\begin{equation*}
\psi_{x_0}(x) = \phi(x_0) + C_{\mathrm{sup}}x_n^{\frac{\alpha}{2}}
\end{equation*}
is a supersolution to the interior equation in the set $\{x_n > 0\}$. We compute that
\begin{equation*}
\Delta \psi_{x_0}(x) = C_{\mathrm{sup}}\frac{\alpha}{2}(\frac{\alpha}{2} - 1)x_{n}^{\frac{\alpha}{2} - 2} < 0
\end{equation*}
while
\begin{equation*}
|\nabla \psi_{x_0}(x)| = C_{\mathrm{sup}}\frac{\alpha}{2}(x_n)^{\alpha - 1},
\end{equation*}
and so $h(\nabla \psi_{x_0}) \notin D$ for $C_{\mathrm{sup}} \geq C$. This shows that $\psi_{C_{\mathrm{sup}}, \phi}$ is a supersolution to the interior equation \eqref{eq interior}. Taking $C_{\mathrm{sup}}$ large enough (depending on the modulus of continuity of $\phi$) also ensures that $\psi$ agrees with $\phi$ on $\partial B_{1}$.

We now turn to the family of subsolutions. For a positive constant $C_{\mathrm{sub}}$ and a small parameter $\beta > 0$, we define
\begin{equation*}
\varphi_{C_{\mathrm{sub}}, \phi}(x) = \max \{\sup_{x_0 \in \partial B_1} \phi(x_0) - C_{\mathrm{sub}}|(x - x_0) \cdot \nu_{x_0}|^{\beta}, \delta\},
\end{equation*}
where $\delta$ is chosen so that $\delta < \min_{x \in \partial B_1} \phi(x)$. To check that $\varphi_{C_{\mathrm{sub}}, \phi}$ is a subsolution of the interior equation \eqref{eq interior} for $C_{\mathrm{sub}} \geq C$, we first note that any positive constant is a subsolution due to the assumption $h(0) < 0$. Next, it suffices to check that the function 
\begin{equation*}
\varphi_{x_0} = \phi(x_0) - C_{\mathrm{sub}}x_n^{\beta}
\end{equation*}
is a subsolution to the interior equation \eqref{eq interior} in the set $\{x_n > 0\}$ as long as $\varphi_{x_0} > \delta$. We compute that
\begin{equation*}
\Delta \varphi_{x_0}(x) = C_{\mathrm{sub}}\beta(1 - \beta)x_n^{\beta - 2},
\end{equation*}
while
\begin{equation*}
\frac{h(\nabla \varphi_{x_0})}{\varphi_{x_0}} \leq \frac{C_h|\nabla \varphi_{x_0}|^2}{\delta} = \frac{C_h C_{\mathrm{sub}}^2\beta^2 x_n^{2\beta - 2}}{\delta}.
\end{equation*}
By a direct computation, we may guarantee that $\varphi_{x_0}$ is a subsolution to the interior equation \eqref{eq interior} as long as
\begin{equation*}
\beta C_{\mathrm{sub}}x_n^{\beta} \leq \frac{(1 - \beta)\delta}{C_h \beta}.
\end{equation*}
Then, notice that
\begin{equation*}
\{\varphi_{x_0} > \delta\} = \{\beta C_{\mathrm{sub}}x_n^{\beta} < \beta(\phi(x_0) - \delta)\}.
\end{equation*}
Taking $\beta$ sufficiently small as a function of $C_{\mathrm{sub}}$ and $\delta$ then ensures that $\varphi_{x_0}$ will be a subsolution to \eqref{eq interior} in the set $\{\varphi_{x_0} > \delta\}$. 

We may then conclude that $\varphi_{C_{\mathrm{sub}}, \phi}(x)$ is a subsolution to \eqref{eq interior}. Additionally, since 
\begin{equation*}
\varphi_{C_{\mathrm{sub}}, \phi}(x) = \sup_{x_0 \in \partial B_1} \phi(x_0) - C_{\mathrm{sub}}|(x - x_0) \cdot \nu_{x_0}|^{\beta} \quad \text{ in a neighborhood of } \partial B_1,
\end{equation*}
it follows that we may choose $C_{\mathrm{sub}}$ depending on the modulus of continuity of $\phi$ so that $\varphi = \phi$ on $\partial B_1$.  
\end{proof}
\begin{remark}
Notice that the family of supersolutions defined in Lemma \ref{lemma sup and sub interior} is monotone increasing in the sense that $\psi_{C_{\mathrm{sup}}, \phi}(x) \rightarrow \infty$ as $C_{\mathrm{sup}} \rightarrow \infty$ for any $x \in B_1$. Likewise, the family of subsolutions is monotone decreasing in the sense that $\varphi_{C_{\mathrm{sub}}, \phi}(x) \rightarrow 0$ as $C_{\mathrm{sub}} \rightarrow \infty$. Hence, any solution $w$ of the interior equation \eqref{eq interior} that agrees with a given boundary datum $\phi$ will satisfy $\varphi_{C_{\mathrm{sub}}, \phi} < w < \psi_{C_{\mathrm{sup}}, \phi}$ for any $C_{\mathrm{sub}}, C_{\mathrm{sup}}$ chosen so that $\varphi_{C_{\mathrm{sub}}, \phi}$ and $\psi_{C_{\mathrm{sup}}, \phi}$ are a subsolution and a supersolution, respectively. Otherwise, we may increase $C_{\mathrm{sup}}$ and $C_{\mathrm{sub}}$ so that a member of the family of supersolutions and subsolutions touches $w$ by above or below (respectively) at some interior point. 
\end{remark}
Now, we may show that given a non-negative boundary datum $\phi$, we may solve the Dirichlet problem  for the interior equation \eqref{eq interior} in the unit ball with $\phi$ as the given boundary datum. We first treat the case where $\phi$ is Holder continuous and strictly positive, and we then handle the general case by approximation.
\begin{proposition}\label{prop existence positive data}
Let $\phi \in C^{\alpha}(\partial B_1)$ be a strictly positive boundary datum. Then, there exists a continuous function $w$ that solves the Dirichlet problem
\begin{align*}
&\Delta w = \frac{h(\nabla w)}{w} \qquad &\text{ in } B_1 \\
&w = \phi \qquad &\text{ on } \partial B_1
\end{align*}
in the viscosity sense.
\end{proposition}
\begin{proof}
We define
\begin{equation*}
\mathcal{A} = \{u \in C(\overline{B_1}): u \text{ is a supersolution to \eqref{eq interior} and $u = \phi$ on $\partial B_1$}\}.
\end{equation*}
We will show that the function 
\begin{equation*}
w(x) = \inf_{u \in \mathcal{A}} u(x)
\end{equation*}
is the desired solution of the Dirichlet problem. To do so, we must first show that $w$ is continuous, which requires us to prove that the minimization can be restricted to a subset of $\mathcal{A}$ that has a uniform modulus of continuity. 

To such end, given $u_1 \in \mathcal{A}$, we construct another supersolution $u_2 \leq u_1$ such that $u_2 \in \mathcal{A}$ and $u_2 \in C^{\beta}(B_1)$ for some $0 < \beta < c$. We do this by using the following Jensen's inf-convolution (see chapter 5 in \cite{caffarelli1995fully})

\begin{equation*}
u_2(x) = \inf_{y \in \overline{B_1}} u_1(y) + C_0|x-y|^{\beta}, 
\end{equation*}
\noindent
where $\beta \leq c$ and $C_0 \geq C$ are parameters to be fixed later. Trivially, we see that $u_2 \leq u_1$; in particular, $u_2 \leq \phi$ on $\partial B_1$. In order to show that $u_2 \geq \phi$ on $\partial B_1$, let $C_{\mathrm{sub}}$ be a positive constant so that the function $\varphi_{C_{\mathrm{sub}}, \phi}$, which was defined in Lemma \ref{lemma sup and sub interior}, is a subsolution of the interior equation \eqref{eq interior}. Then, since $u_1 \geq \varphi_{C_{\mathrm{sub}}, \phi}$ and 
\begin{equation*}
\varphi_{C_{\mathrm{sub}}, \phi}(x) \geq \phi(x_0) - C_{\mathrm{sub}}'|x - x_0|^{\frac{\alpha}{2}},    
\end{equation*}
for some constant $C_{\mathrm{sub}}'$, then for $y \in B_1$, $x_0 \in \partial B_1$

\begin{equation*}
u_1(y) + C_0|x_0 - y|^{\beta} \geq \phi(x_0) + C_0|x_0 - y|^{\beta} - C_{\mathrm{sub}}'|x_0 - y|^{\frac{\alpha}{2}} \geq \phi(x_0).
\end{equation*}
\noindent
Hence, as long as we choose $\beta < \frac{\alpha}{2}$ and $C_0$ large enough depending on $C_{\mathrm{sub}}'$, we can guarantee that $u_2 = \phi$ on $\partial B_1$.

Now, we will show that $u_2$ is a supersolution to the interior equation \eqref{eq interior}. Suppose that $P$ is a polynomial that touches $u_2$ by below at $x_0 \in B_1$. Let $y_0 \in \overline{B_1}$ be such that

\begin{equation*}
u_2(x_0) = u_1(y_0) + C_0|x_0 - y_0|^{\beta}.
\end{equation*}
We must first check that $y_0 \notin \partial B_1$; assume, by way of contradiction, that $y_0 \in \partial B_1$. Then, we let $C_{\mathrm{sup}}$ be a constant such that the function $\psi_{C_{\mathrm{sup}}, \phi}$ defined in Lemma \ref{lemma sup and sub interior} is a supersolution to the interior equation. Then, since 
\begin{equation*}
\psi_{C_{\mathrm{sup}}, \phi}(x_0) \leq \phi(y_0) + C_{\mathrm{sup}}'|x - x_0|^{\frac{\alpha}{2}}, 
\end{equation*}
for some constant $C_{\mathrm{sup}}'$, we get that
\begin{equation*}
\phi(y_0) + C_0|x_0 - y_0|^{\beta} = u_1(y_0) + C_0|x_0 - y_0|^{\beta} = u_2(x_0) \leq u_1(x_0) \leq \phi(y_0) + C_{\mathrm{sup}}'|x_0 - y_0|^{\frac{\alpha}{2}},
\end{equation*}
which yields a contradiction provided we take $\beta < \frac{\alpha}{2}$ and $C_0$ large enough depending on $C_{\mathrm{sup}}'$. 

Hence, without loss of generality, we may take $x_0 = 0$ and $y_0 = te_n$ for some $t \in (0,1)$. Then, we have that

\begin{equation*}
u_2(0) = u_1(te_n) + C_0|te_n|^{\beta}.
\end{equation*}
Since in general

\begin{equation*}
u_2(x) \leq u_1(te_n) + C_0|x - te_n|^{\beta},
\end{equation*}
we may immediately deduce that

\begin{equation*}
|\nabla P(0)| = C_0\beta t^{\beta - 1}.
\end{equation*}
\noindent
Given that $t^{\beta - 1} \geq c$, we may then take $C_0$ large enough as a function of $\beta$ so that $\nabla P(0) \notin D$.

Now, letting $y = (y', y_n)$ for $y' \in \mathbb{R}^{n-1}$ and $y_n \in \mathbb{R}$, the function

\begin{equation*}
v(y) = P(y', 0) - C_0|y_n|^{\beta}
\end{equation*}
\noindent
touches $u_1$ by below at $te_n$. Furthermore, notice that

\begin{equation*}
\Delta v(0) = \Delta P(te_n) + C_0(\beta)(1-\beta)|y_n - t|^{\beta-2} - P_{nn}(te_n). 
\end{equation*}
\noindent
The touching then yields 
\begin{align*}
&|\nabla u_1(0)| = C_0\beta t^{\beta - 1} \\
&\Delta u_1(0) \geq \Delta P(te_n) + 2C_0\beta(1 - \beta)t^{\beta - 2} \\
& P_{nn}(te_n) \leq C_0\beta (1-\beta)|t|^{\beta}.   
\end{align*}
\noindent
Additionally, we have that for $C_1 = \inf_{y \in B_1} u(y)$
\begin{equation*}
\Delta u_1(0) \leq \frac{h(\nabla u_1(0))}{u_1(0)} \leq \frac{C_h|\nabla u_1(0)|^{2}}{u_1(0)} \leq \frac{C_0^{2}C_h\beta^2 t^{2\beta - 2}}{C_1}. 
\end{equation*}
\noindent
Consequently, we get that
\begin{align*}
\Delta P(te_n) &\leq \frac{C_hC_0^{2}\beta^2 t^{2\beta - 2}}{C_1} - 2C_0\beta(1 - \beta)t^{\beta - 2} \\
&=\beta C_0t^{\beta - 2}[\frac{C_hC_0\beta t^{\beta} - 2(1 - \beta)C_1}{C_1}].
\end{align*}
Then, we notice that since
\begin{equation*}
C_0t^{\beta} = u_2(x_0) - u_1(y_0) \leq u_1(x_0) - u_1(y_0) \leq ||u_1||_{L^{\infty}(B_1)} \leq C, 
\end{equation*}
we may then take $\beta < c$ so that 
\begin{equation*}
\Delta P(te_n) < 0. 
\end{equation*}
Thus,
\begin{equation*}
\Delta P(te_n) \leq \frac{h(\nabla P(te_n))}{P(te_n)}.    
\end{equation*}
\noindent
Hence, we can conclude that 
\begin{equation*}
w(x) = \inf_{u \in A} u(x) 
\end{equation*}
\noindent
is a continuous supersolution. To show that it is also a subsolution, let us suppose by contradiction that there exists a polynomial $P$ that touches $w$ strictly by above at $x_0$ such that 
\begin{equation*}
\Delta P(x_0) < \frac{h(\nabla P(x_0))}{P(x_0)}.
\end{equation*}
\noindent
Then, the function
\begin{equation*}
\bar{w} = 
\begin{cases}
w \quad \text{ for $x \in B_{1} \setminus \bar{B}_{\rho}$} \\
\min\{P - c|x - x_0|^{\beta_0}, w\} \quad \text{ for $x \in B_{\rho}$}
\end{cases}
\end{equation*}
is, for any choice of $\beta_0 < 1$, a supersolution that is strictly less than $w$ for $\rho$ small enough, a contradiction.  
\end{proof}
\begin{corollary}
Let $\phi \in C(\partial B_1)$ be a non-negative boundary datum. Then, there exists a solution of the interior equation \eqref{eq interior} that agrees with $\phi$ on $\partial B_1$. 
\end{corollary}
\begin{proof}
We proceed by an approximation argument. For $\epsilon > 0$, consider a family of mollifications $\{\phi_{\epsilon}\}_{\epsilon \in \mathbb{R}}$ such that $\phi_{\epsilon} \in C^{\alpha}(\partial B_1)$ for some $\alpha$ depending on $\epsilon$. Then, consider the sequence of solutions $w_{\epsilon}$ of the Dirichlet problem for the interior equation with boundary datum $\phi_{\epsilon}$. We then recall the regularity estimates

\begin{equation*}
||w_{\epsilon}||_{C^{2, \alpha}(B_{\frac{1}{2}})} \leq C||w_{\epsilon}||_{L^{\infty}(B_1)}.
\end{equation*}
Hence, by passing to a subsequence if necessary, we may conclude that $w_{\epsilon} \rightarrow w$ uniformly for any $K \subset \subset \Omega$. 

Now, to ensure convergence at the boundary, we use the existence of the family of supersolutions and subsolutions constructed in Lemma \ref{lemma sup and sub interior}. Namely, suppose by way of contradiction that $w(x_0) = K$ for some $K \geq 0$ and $x_0 \in \partial B_1$ and that there exists some fixed $\epsilon > 0$ and a sequence of points $x_{j} \rightarrow x_0$ such that 
\begin{equation*}
w_{k}(x^{j}) \geq K + \epsilon    
\end{equation*}
for all $k \geq C(x^{j})$. Then, we consider a family of supersolutions $\{v_{j}\}$ to the interior equation \eqref{eq interior} in $B_1$ such that $\{v_j\} \geq \phi$ on $\partial B_1$ and $v_{j} \rightarrow \infty$ pointwise. Then, some $v_j$ will touch some $w_{k}$ by above at an interior point, providing a contradiction. Likewise, if $K > 0$ and 
\begin{equation*}
w_{k}(x^{j}) \leq K - \epsilon    
\end{equation*}
for all $k \geq C(x^{j})$, then we consider a sequence of balls $B_k$ that are interior to $B_1$, are tangent to $\partial B_1$ at $x_0$, and such that $x^{j} \in B_{j}$. Then, for each ball $B_{j}$, we construct a family of subsolutions $\{u_{k}^{j}\}_{k}$ such that $u_{k}^{j} \leq w$ on $\partial B_j$ (notice we can do this since the function is strictly positive on $B_j$) and $u_{k}^{j} \rightarrow 0$ pointwise as $k \rightarrow \infty$. It then follows that some $u_{k}^{j}$ will touch some $w_i$ by below at an interior point, providing a contradiction. 

Thus, we conclude that $w$ is the desired solution of the Dirichlet problem. Moreover, $w$ satisfies the regularity estimate \eqref{eq interior regularity}. 
\end{proof}
The previous result then concludes the proof of Theorem \ref{thm interior equation}. We conclude our discussion of the interior equation by showing that imposing an additional monotonicity assumption on the non-linearity $h$ yields a uniqueness result. This extends the corresponding result known for the Alt-Phillips equation for negative power potentials.

\begin{proposition}
Suppose that the right-hand side of the interior equation \eqref{eq interior} satisfies that the map $t \rightarrow h(t\nu)$ is strictly increasing for all $\nu \in \mathbb{S}^{n-1}$. Then, for any $\phi \in C(\partial B_1)$, the solution of the Dirichlet problem for the interior equation \eqref{eq interior} in $B_1$ with boundary datum $\phi$ admits a unique solution.
\end{proposition}
\begin{proof}
Suppose, by way of contradiction, that $w_1$ and $w_2$ both solve the Dirichlet problem with boundary datum $\phi$ and that $w_1 \neq w_2$. For $i \in \{1,2\}$, we then let
\begin{equation*}
u_i = \log w_i,
\end{equation*}
noting that $u_i$ then satisfies the interior equation
\begin{equation*}
\Delta u_i + |\nabla u_i|^2 = h(e^{u_i}\nabla u_i).
\end{equation*}
Then, since $w_1 \neq w_2$, we may assume by way of contradiction that $u_1 - u_2$ attains a positive maximum at the origin. Since $\nabla u_1(0) = \nabla u_2(0)$, we then get that for $u = u_1 - u_2$,
\begin{equation*}
\Delta u(0) = h\big(e^{u_1(0)}\nabla u_1(0)\big) - h\big(e^{u_2(0)}\nabla u_2(0)\big) > 0,
\end{equation*}
where in the last inequality we have used the additional assumption of the monotonically increasing property of $h$. This then yields a contradiction to $\Delta u(0) \leq 0$ due to $0$ being a maximum.
\end{proof}

\subsection{The Free Boundary Problem}\label{subsec free boundary}

Now, we treat solutions to the free boundary problem \eqref{eq main problem}. Using the Harnack inequality along with the free boundary condition \eqref{eq free boundary}, we first prove an a priori optimal growth result near free boundary points, which directly yields the Lipschitz regularity of solutions near free boundary points specified in Theorem \ref{thm Lipschitz Regularity}
\begin{proposition}\label{prop optimal growth}
Let $w$ be a solution to \eqref{eq main problem}. Then, 
\begin{equation}\label{eq optimal growth}
w(x) \leq Cdist(x, FB(w)).
\end{equation}
\end{proposition}
\begin{proof}
We proceed by contradiction. As in the previous result, we may assume (by employing a Lipschitz rescaling) that $0 \in FB(w)$ and that $B'$ is a ball tangent to the free boundary at the origin. Next, for $\epsilon, \mu > 0$, consider the function 
\begin{equation}
v(x) = f(\nu(x))d(x) + \epsilon d(x)^{1 + \mu},
\end{equation}
where
\begin{align*}
d(x) = 
\begin{cases}
dist(x, \partial B') \qquad & \text{ for } x \in B' \\
0 \qquad & \text{ otherwise,}
\end{cases}
\end{align*}
and $\nu(x) = \nabla d(x)$. Per Lemma \ref{Lemma Test Functions}, choose $\mu$ so that $v$ is a subsolution to \eqref{eq main problem} in a neighborhood of $\partial B'$. Then, we let $B'' \subset B'$ be a ball with the same center as $B'$ such that $v$ is a subsolution to \eqref{eq main problem} in $B' \setminus B''$. Then, if $w \geq C$ at the center of $B'$, we may use the Harnack inequality \eqref{eq Harnack} with $\sigma = 0$ to deduce that $w \geq C$ in $\partial B''$. We may then guarantee that $v < w$ on $\partial B''$. Then, if $v - w$ attains a positive maximum in $B' \setminus B''$, then the function $\bar{v}$ defined in the same way as $v$ but with $d(x)$ now denoting the distance to a ball $B^{'''}$ that contains $B^{''}$ and lies in $B^{'}$ will for an appropriate choice of $B^{'''}$ touch $w$ by below at an interior point, providing a contradiction. Hence, we conclude that $v$ touches $w$ by below at the origin, a contradiction to the free boundary condition \eqref{eq free boundary}.
\end{proof}

We are now in a position to prove the existence of solutions to the free boundary problem \eqref{eq main problem} with any given non-negative continuous boundary data. We proceed by using Perron's method along with a replacement procedure that uses the fact that we may solve the Dirichlet problem for the interior equation \eqref{eq interior} in $B_1$. Before proceeding, we clarify for the reader that a supersolution (resp. subsolution) to the free boundary problem \eqref{eq main problem} is a supersolution (resp. subsolution) to the interior equation \eqref{eq interior} that satisfies the supersolution free boundary condition \eqref{eq free boundary condition supersol} (resp. subsolution free boundary condition \eqref{eq free boundary condition subsol}).

\begin{theorem}
Let $\phi \in C(\partial B_1)$ be non-negative. Then, there exists a solution to the Dirichlet free boundary problem 
\begin{align*}
&\Delta w = \frac{h(\nabla w)}{w} && \text{in } B_1^{+}, \label{eq main problem}\\
&w = \phi && \text{on } \partial B_1, \nonumber \\
&\partial^{-}_{\nu \nu} w = 0 && \text{in } FB(w). \nonumber
\end{align*}
\end{theorem}
\begin{proof}
We define the class of functions
\begin{equation*}
\mathcal{A} = \{v \in C(\overline{B_1}): \text{ is a supersolution to \eqref{eq main problem} and $v = \phi$ on $\partial B_1$}\}.
\end{equation*}
We then set
\begin{equation*}
w(x) := \inf_{v \in A} v(x).
\end{equation*}
\noindent
We must first show that $w$ is continuous. To such end, we will demonstrate that the positivity set of $w$ is open. 

Suppose that $w(0) > 0$. We claim that every $v \in \mathcal{A}$ is strictly positive in $B_{c_0w(0)}$ for some small universal parameter $c_0$. Let $v \in \mathcal{A}$, and suppose by way of contradiction that we can find some $x_0 \in B_1$ with $v(x_0) = 0$ and $|x_0| < c_0w(0)$. By a Lipschitz rescaling of factor $|x_0|^{-1}$, we may suppose that $v > 0$ in $B_1$, that $x_0 \in \partial B_1$, and that $w(0) > c_0^{-1}$. Now, consider the test function
\begin{equation*}
u(x) = f(\nu(x))d(x) + d(x)^{1 + \mu},
\end{equation*}
where $d(x)$ is the distance to the boundary of $B_1$, $\nu(x) = \nabla d(x)$, and $\mu$ is such that $u$ is a subsolution to \eqref{eq main problem} in $B_{1 - c(\mu)}$ by Lemma \ref{Lemma Test Functions}. Next, we may replace $v$ in $B_{1 - \frac{c(\mu)}{2}}$ with the minimal Perron solution of the interior equation \eqref{eq interior}; notice that $v$ still remains in the competition class $\mathcal{A}$ after this transformation and is smaller than the original $v$ that was considered. Now, we may apply the Harnack inequality in Lemma \ref{Lemma Harnack Inequality} with $\sigma = 0$  to $v$ to deduce that $v \geq c \cdot w(0)$ in $B_{1 - c(\mu)}$; hence, 
\begin{equation}\label{eq harnack perron fb}
w \geq c \cdot w(0).
\end{equation}

If $c_0$ is sufficiently small, then $v > u$ on $\partial B_{1 - c(\mu)}$. We note that $v - u$ cannot attain a negative minimum in $B_1 \setminus B_{1 - c(\mu)}$; otherwise, a test function of the form described in Lemma \ref{Lemma Test Functions} with $d(x)$ now being the distance to $\partial B_{1 - t}$ for $t > 0$ would touch $v$ by below at an interior point, yielding a contradiction. We thus conclude that $u$ touches $v$ by below at the origin, providing a contradiction and thus proving the claim above. 

In conclusion, \eqref{eq harnack perron fb} now shows that $w > 0$ on $B_{c_0w(0)/2}$. This then shows that the positivity set of $w$ is open, and so $w$ solves the interior equation \eqref{eq interior} in its positivity set.  

It remains to show that $w$ satisfies the free boundary condition \eqref{eq free boundary}. So, assume that $0 \in FB(w)$, and let $\Gamma$ be a $C^{3}$ hypersurface such that $0 \in \Gamma$. First, suppose, by way of contradiction, that the test function
\begin{equation*}
u(x) = f(\nu(x))d(x) + sd(x)^{1 + \mu}
\end{equation*}
touches $w$ by below at the origin, where $d$ is the distance to $\Gamma$ on one side of $\Gamma$ and $0$ on the other, $\nu(x) = \nabla d(x)$, $\mu \in (0,1)$, and $s > 0$. Without loss of generality, we may assume $\mu$ satisfies \eqref{eq mu full condition} and $s = 1$. Then, consider the two parallel surfaces of $\Gamma$ that are at distance $\epsilon_0$ from $\Gamma$, where the parameter $\epsilon_0 > 0$ is to be precised later. We will construct a new surface $\Gamma_1$ in a small ball $B_{c}$ as follows: $\Gamma_1$ coincides with the parallel surface to $\Gamma$ that lies in the set $\{d = 0\}$ in $B_{\frac{c}{4}}$, $\Gamma_1$ coincides with the other parallel surface of $\Gamma$ in $B_{c} \setminus B_{\frac{3c}{4}}$, and $\Gamma_1$ is a smooth deformation between the two parallel surfaces in $B_{\frac{3c}{4}} \setminus B_{\frac{c}{4}}$. We now define
\begin{equation*}
u_1(x) = f(\nu_1(x))d_1(x) + \delta d_1(x)^{1 + \mu},
\end{equation*}
where $d_1(x)$ is the distance to $\Gamma'$ on the side of $\Gamma'$ intersecting the positivity set of $d$, $\nu_1(x) = d_1(x)$, and the parameters $\mu, \delta,$ and $\epsilon_0$ are chosen so that $u_1$ is a subsolution of the interior equation \eqref{eq interior} in $B_c$ and so that $u_1 < u$ on $\partial B_c$. It then follows that an appropriate choice of $\epsilon_0$ implies that $u_1$ touches some supersolution $v \in \mathcal{A}$ that lies close to $w$ from below, providing a contradiction. 

Finally, suppose, by way of contradiction, that the test function
\begin{equation*}
u(x) = f(\nu(x))d(x) -  sd(x)^{1 + \mu}
\end{equation*}
touches $w$ by above at the origin. As in the previous paragraph, we consider two parallel copies of $\Gamma$ that are separated by a distance $\epsilon_0$. However, we now choose the new surface $\Gamma'$ so that it coincides with the parallel surface that lies in the set $\{d > 0\}$ near the origin. Specifically, $\Gamma'$ coincides with such surface in $B_{\frac{c}{4}}$, coincides with the other surface in $B_{c} \setminus B_{\frac{3c}{4}}$, and is a smooth deformation between the parallel surfaces in $B_{\frac{3c}{4}} \setminus B_{\frac{c}{4}}$. Similarly to our previous reasoning, we now define
\begin{equation*}
u_1(x) = f(\nu_1(x))d_1(x) - \delta d_1(x)^{1 + \mu},
\end{equation*}
where $d_1(x)$ is now the distance to $\Gamma'$ on the side of $\Gamma'$ intersecting the positivity set of $d$, $\nu_1(x) = \nabla d_1(x)$, and the parameters $\mu, \delta,$ and $\epsilon > 0$ are then chosen small enough so that $u_1 \geq u$ on $\partial B_c$. Note that the origin lies in the interior of the zero set of $u_1$, and that $u_1 > w$ on $\partial B_c$. Thus, it follows that the function
\begin{equation*}
\bar{w} = 
\begin{cases}
\min\{w, u_1\} \qquad &\text{ in $B_c$} \\
w \qquad &\text{elsewhere}
\end{cases}
\end{equation*}
satisfies $\bar{w} \in \mathcal{A}$ and $\bar{w} < w$, a contradiction.

\end{proof}

We conclude this section with a non-degeneracy estimate for Perron solutions of the free boundary problem \eqref{eq main problem}, which will complete the proof of Theorem \ref{thm Perron Existence}. The proof crucially uses the fact that the given solution can be seen as the least supersolution; in general, a solution to \eqref{eq main problem} may not satisfy this estimate. 

\begin{proposition}[Non-degeneracy]
For a given continuous, non-negative boundary datum $\phi$, let $w$ be the Perron solution of the Dirichlet problem for the free boundary problem \eqref{eq main problem}. For $x_0 \in FB(w)$ and $r$ such that $B_{r}(x_0) \subset \subset B_1$, it follows that
\begin{equation*}
\sup_{x \in B_r(x_0)} w(x) \geq cr.
\end{equation*}
\end{proposition}
\begin{proof}
Suppose, by way of contradiction, that the conclusion is false. Then, there exists some $x_0 \in FB(w)$, $r > 0$, and some $\delta \leq c$ such that
\begin{equation*}
w(x) \leq \delta r
\end{equation*}
for any $x \in \partial B_r(x_0)$. Then, for $\epsilon, t \leq c$ and $\mu > 0$, let
\begin{equation*}
v(x) = f(\nu(x))d(x) - \epsilon d(x)^{1 + \mu},
\end{equation*}
where 
\begin{align*}
d(x) = 
\begin{cases}
dist(x, B_{r - t}(x_0)) \qquad &\text{for $x \in B_{r}(x_0)$}\\
0 \qquad &\text{ otherwise},
\end{cases}
\end{align*}
$\nu(x) = \nabla d(x)$, and $\mu$ is chosen so that $v$ is a supersolution to \eqref{eq main problem} in a neighborhood of $\partial B_{r - t}(x_0)$ per Lemma \ref{Lemma Test Functions}. Then, it follows that upon making $\delta, t$ small enough, we may guarantee that $v \geq w$ on $B_{r}(x_0)$. As a result, the function
\begin{equation*}
\bar{w} = 
\begin{cases}
0 \quad \text{ for } x \in B_{r - t}(x_0) \\
\min\{v, w\} \quad \text{ for } x \in B_{r}(x_0) \setminus B_{r-t}(x_0) \\
w \quad \text{ otherwise}
\end{cases}
\end{equation*}
is a supersolution to the Dirichlet problem for the free boundary problem \eqref{eq main problem} with boundary datum $\phi$ that is strictly less than $w$, a contradiction.
\end{proof}

\begin{remark}\label{rmrk convergence for fixed right-hand side}
We note that, due to the uniform growth and non-degeneracy bounds, a sequence of solutions $\{w_k\}$ to \eqref{eq main problem} such that each solution $w_k$ satisfies the interior equation with a fixed right-hand side $h$ and a fixed star-shaped domain $D$ admits a convergent subsequence to a limiting continuous function $w$, which is itself a viscosity solution to the interior equation \eqref{eq interior} in $B_1$. Moreover, the non-degeneracy estimate guarantees that the free boundaries $FB(w_k)$ converge to $FB(w)$ in Hausdorff distance.
\end{remark}
\section{Improvement of Flatness}\label{sec flatness improv}
In this section, we show the improvement of flatness argument at regular free boundary points given in Theorem \ref{thm flatness improv}. We mostly follow the strategy of De Silva in \cite{de2011free}. The first step in such argument is usually known as a semi-Harnack inequality, which we present below.  
\begin{lemma}\label{Lemma SemiHarnack}
Let $w$ be a solution to \eqref{eq main problem} in $B_1$ with $e_n \in \partial D$. Then, there exists a universal parameter $\epsilon \leq c$ such that the following claims hold.
\begin{enumerate}
    \item If the one-sided flatness assumption
    \begin{equation}\label{eq flatness 1}
    w \geq x_n \qquad \text{ in $B_1$}
    \end{equation}
    holds and 
    \begin{equation*}
    w = x_n + \epsilon \qquad \text{ at $x_0 = \frac{e_n}{2}$},
    \end{equation*}
    then there exists some $r > 0$ universal such that
    \begin{equation*}
    w \geq x_n + c\epsilon \qquad \text{ in $B_r$.}
    \end{equation*}
    \item If the one-sided flatness assumption
    \begin{equation}\label{eq flatness 2}
    w \leq (x_n + 2\epsilon)_{+} \qquad \text{ in $B_1$}
    \end{equation}
    holds and 
    \begin{equation*}
    w = x_n + \epsilon \qquad \text{ at $x_0 = \frac{e_n}{2}$},
    \end{equation*}
    then there exists some $r > 0$ universal such that
    \begin{equation*}
    w \leq x_n + c\epsilon \qquad \text{ in $B_r$.}
    \end{equation*}
\end{enumerate}
\end{lemma}
\begin{proof}
We prove the first claim; the second one follows analogously. First, for a positive universal constant $\sigma \leq c$, consider the set

\begin{equation*}
\mathcal{C}= \{x = (x',x_n): |x'|^{2} \leq \frac{1}{2} \text{ and } x_n \geq \sigma\}. 
\end{equation*}
\noindent
Then, we notice that by the flatness assumption \eqref{eq flatness 1},
\begin{equation*}
\sigma \leq w \leq C, \qquad |\nabla w| \leq C,     \qquad \text{ in $\mathcal{C}$}.
\end{equation*}
As a result, the interior equation \eqref{eq interior} can be written as a uniformly elliptic equation with bounded, measurable coefficients in $\mathcal{C}$. Hence, the classical Harnack inequality applies, and we may thus conclude that

\begin{equation}\label{eq flat Harnack}
w \geq x_n + c(\sigma)\epsilon \quad \text{ in } \mathcal{C}. 
\end{equation}

Now, we prove a similar inequality when $|x_n| \leq \sigma$. For a large radius $R$, let $\tilde{B}$ be the circle of radius $R$ centered at $Re_n$, and let 
\begin{equation*}
\mathcal{C}' = \tilde{B} \cap \{|x_n| \leq \sigma, |x'|^2 \leq \frac{1}{2}\}.    
\end{equation*}
Let 

\begin{align*}
d(x) = 
\begin{cases}
dist(x, \partial \tilde{B}) \qquad & \text{ for } x \in \tilde{B} \\
0 \qquad & \text{ otherwise},
\end{cases}
\end{align*}
and $\nu(x) = \nabla d(x)$. Now, we will set $R = \frac{M}{\epsilon c_1}$, where $M$ and $c_1$ are large and small, respectively, universal parameters to be precised later in the proof. Next, we choose an exponent $\mu$ as in Lemma \eqref{Lemma Test Functions} so that the function
\begin{equation*}
u(x) = f(\nu(x))d(x) + c_1\epsilon d(x)^{1 + \mu}, \end{equation*}
is a subsolution to the interior equation \eqref{eq interior} in a neighborhood of $\tilde{B}$. Our goal is to show that there exists another small universal parameter $c_2$ such that 
\begin{equation*}
w(x) \geq u(x + c_2e_n)     
\end{equation*}
on $\mathcal{C'}$. First, choose $M$ large enough as a function of $c_1$ so that
\begin{equation}\label{eq parameter M}
1 - c_1\epsilon \leq f(\nu(x)) \leq 1 + c_1\epsilon \qquad \text{ on $\mathcal{C'}$.} 
\end{equation}
Now, we show that our parameters can be chosen so that $w > u$ on $\partial \mathcal{C}'$. First, consider the part of $\partial \mathcal{C'}$ where $|x'|^{2} = \frac{1}{2}$. By a direct computation, it may be checked that
\begin{equation*}
\{(x', x) : x_n = \frac{|x'|^2}{2R}\} \cap \tilde{B} = \emptyset.    
\end{equation*}
Hence, we may conclude that
\begin{equation}\label{eq d on x'}
d(x) \leq (x_n - \frac{|x'|^{2}}{2R})_{+}. 
\end{equation}
Using this and the flatness assumption \eqref{eq flatness 1}, 
we will show that
\begin{equation*}
u \leq w - c_2 \qquad \text{   on   } \mathcal{C} \cap \{|x'|^2 = \frac{1}{2}\}.   
\end{equation*}
 Indeed, applying \eqref{eq parameter M} and \eqref{eq d on x'},
\begin{equation*}
u(x) \leq (1 + c_1\epsilon)(x_n - \frac{c_1}{4M}) + c_1\epsilon(x_n - \frac{c_1}{4M})^{1 + \mu} \leq (1 + 2c_1\epsilon)(x_n - \frac{c_1}{4M}).
\end{equation*}
on $\mathcal{C'} \cap |x'|^2 = \frac{1}{2}$. At the same time, the following holds on such region as well due to the flatness assumption \eqref{eq flatness 1}
\begin{equation*}
w(x) - c_2 \geq x_n - c_2.
\end{equation*}
Hence, we will have $u \leq w - c_1$ as long as 
\begin{equation*}
c_2 \leq \frac{c_1}{4M} + \frac{2c_1^2\epsilon}{4M} - 2c_1\epsilon x_n.
\end{equation*}
Note then that on this region, $x_n \leq \frac{c_1}{4M}$, and so the above equation simplifies to
\begin{equation}\label{eq c_2 condition 1}
c_2 \leq \frac{c_1}{4M}. 
\end{equation}

Now, we consider $\mathcal{C'} \cap \{x_n = \sigma\}$. Here, we simply use the fact that
$d(x) \leq x_n^{+}$ and the Harnack inequality \eqref{eq flat Harnack}. On this region, we then have that
\begin{equation*}
u(x) \leq (1 + 2c_1\epsilon)x_n
\end{equation*}
and that 
\begin{equation*}
w(x) - c_2\geq x_n + c(\sigma)\epsilon - c_2.
\end{equation*}

So, we will have $u \leq w - c_1$ in $\mathcal{C}' \cap \{x_n = \sigma\}$ as long as
\begin{equation}\label{eq c_2 condition 2}
c_2 \leq c(\sigma)\epsilon - 2c_1\epsilon. 
\end{equation}
In conclusion, \eqref{eq c_2 condition 1} and \eqref{eq c_2 condition 2} yield $u \leq w - c_2$ on $\partial \mathcal{C}'$ as long as 
\begin{equation}\label{eq c_2 condition}
c_2 \leq \min\{\frac{c_1}{4M}, c(\sigma)\epsilon - 2c_1\epsilon\}.
\end{equation}
Hence, upon setting
\begin{equation*}
c_1 \leq \frac{c(\sigma)}{4}, M = M(c_1),
\end{equation*}
and setting $\sigma$ as a function of $\epsilon$ so that $u$ is indeed a subsolution in $\mathcal{C'}$, we can ensure that
\begin{equation*}
\min\{\frac{c_1}{4M}, c(\sigma)\epsilon - 2c_1\epsilon\} > 0,
\end{equation*}
and so $c_2$ may be chosen as in \eqref{eq c_2 condition}.

Now that it is known that $u \leq w - c_2$ on $\partial \mathcal{C'},$ we may compare $w$ with a sequence of translations of $u$, namely $u(x + te_n)$ for $t \leq c_2$, and conclude that $w(x) \geq u(x + c_2e_n)$ for $x \in \mathcal{C}'$. This then implies the desired inequality in a neighborhood of the origin. 
\end{proof}

\begin{remark}\label{remark consequence of semiharnack}
Let $\{w_j\}$ be any sequence of solutions to \eqref{eq main problem} that satisfy the flatness assumption with an associated sequence of flatness parameters $\{\epsilon_j\}$. We now explain why the sequence of rescaled functions
\begin{equation*}
\bar{w_j} := \frac{w_j - x_n}{\epsilon_j}    
\end{equation*}
has the required compactness to converge to a limiting function $\bar{w}$ solving a linearized free boundary problem. It follows from standard arguments as in \cite{de2021certain} that we may use Lemma \ref{Lemma SemiHarnack} to show an improvement of flatness at fixed scale: namely, if 
\begin{equation*}
(x_n - a_1)_{+} \leq w \leq (x_n + a_2)_{+} \quad \text{ in $B_1$}
\end{equation*}
for some $a_1, a_2 > 0$ and
\begin{equation*}
|a_2 - a_1| \leq c,
\end{equation*}
then
\begin{equation*}
(x_n - (1 - c)a_1)_{+} \leq w \leq (x_n + (1 - c)a_2)_{+} \quad \text{ in $B_c$}.
\end{equation*}
We may then iterate this to gain the desired compactness. Indeed, it follows from the improvement of flatness at fixed scale (see Remark 5.2 in \cite{de2021certain}) that the sequence of rescalings $\bar{w_j}$ has a uniform modulus of continuity in the set $|y - x| \geq \delta$ (where such modulus depends on $\delta$).This then implies the uniform convergence of such sequence of rescaled functions to a limiting function $\bar{w}$ solving a certain free boundary problem, which we detail in the next result. 
\end{remark}

As in \cite{de2021certain}, the strategy for proving Theorem \ref{thm flatness improv} will be to show that, if the conclusion of such result is not true, then the rescaled solutions $\bar{w_j}$, which already satisfy the interior equation
\begin{equation}
\Delta \bar{w_j} = \frac{h(e_n + \epsilon_j\nabla w_j)}{\epsilon_j(x_n + \epsilon_jw_j)},
\end{equation}
\noindent
will converge to a solution $\bar{w}$ of a linearized Neumann problem of the type
\begin{align}\label{eq linearized}
\begin{cases}
\Delta \bar{w} = \frac{v \cdot \nabla h(e_n)}{x_n} \quad &\text{ in } B_{1}^{+} \cap \{x_n > 0\}, \\
\bar{w}_{\omega,s} := \frac{\bar{w}(x_0 + t\omega) - \bar{w}(x_0)}{t^{1-s}} = 0 &\text{ on } B_1^{+} \cap \{x_n = 0\},
\end{cases}
\end{align}
where $v$ is a fixed non-zero vector and $\omega = \frac{v}{|v|}$. In our case, $s = -\nabla h(e_n) \cdot e_n < 1$ by assumption \eqref{eq mu condition} and $v = \nabla h(e_n)$. In the regime $s < 1$, $\bar{w}$ being a viscosity solution to \eqref{eq linearized} entails the following:
\begin{enumerate}
    \item $\bar{w}$ is continuous in $B_1$ and is a viscosity solution to the interior equation wherever it is positive
    \item $\bar{w}$ may not be touched from below $(+)$ or by above $(-)$ at a point $x_0 \in B_1^{+} \cap \{x_n = 0\}$ by a function of the form
    \begin{equation*}
    \phi(x',x_n) = A|x' - \frac{\omega'}{\omega_n} x_0 - y_0'|^{2} + B \pm px_n^{1-s}
    \end{equation*}
    for any fixed $A, B \in \mathbb{R}$, $p > 0$ and $y_0' \in \mathbb{R}^{n-1}$ ($y_0'$ is such that $y_0' \cdot e_n = 0$). 
\end{enumerate}
\begin{proof}[Proof of Theorem \ref{thm flatness improv}]
Suppose by way of contradiction that there exists a sequence of solutions $\{w_j\}$ to \eqref{eq main problem}, along with an associated sequence of star-shaped domains $\{D_j\}$ and an associated sequence of right-hand sides $\{h_j\}$ that satisfy Assumption \ref{assumption standard assumption},  as well as a sequence $\{\epsilon_j\}$ converging to $0$ such that the conclusion is not true. We may assume that $\{w_j\},\{h_j\},$ and $\{D_j\}$ have uniform bounds that allow, up to the extraction of a subsequence, 
\
\begin{equation*}
w_j \rightarrow w \quad h_j \rightarrow h \quad  D_j \rightarrow D.
\end{equation*}
\noindent
Then, consider the sequence of rescalings 
\begin{equation*}
\bar{w_j} = \frac{w_j - x_n}{\epsilon_j}.
\end{equation*}
Remark \ref{remark consequence of semiharnack} implies that $\bar{w_j}$ converges, up to a subsequence, to a limiting function $\bar{w}$ that satisfies $\bar{w}(0) = 0$ and $|\bar{w}| \leq 1$. We then compute that

\begin{equation}
\Delta \bar{w_j} = \frac{h(e_n + \epsilon_j\nabla \bar{w_j})}{\epsilon_j(x_n + \epsilon_j\bar{w}_j)}.
\end{equation}
\noindent
Then, since we are assuming that $e_n \in \partial D_j$ for all $j$, we see that
\begin{equation*}
h_j(e_n + \epsilon_j\nabla \bar{w}_j) = \epsilon_j\nabla h_j(e_n) \cdot\nabla \bar{w}_j + o(\epsilon_j).
\end{equation*}
Consequently, we can conclude that $\bar{w}$ is a viscosity solution in its positivity set to the interior equation appearing in \eqref{eq linearized}. Now, we must show it satisfies the free boundary condition in \eqref{eq linearized}. For simplicity, assume $\bar{w}(0) = 0$ and that it can be touched from below by
\begin{equation*}
\phi(x) = A|x' - \frac{\omega'}{\omega_n}x_n|^2 + px_n^{1-s}
\end{equation*}
for some $A, p > 0$. For a parameter $\epsilon > 0$, let $B'$ be the ball of radius $\frac{1}{2\epsilon A}$ centered at $\frac{e_n}{2\epsilon A}$. As in the previous Lemma \ref{Lemma SemiHarnack}, let 
\begin{align*}
d(x) = 
\begin{cases}
dist(x, \partial B') \qquad & \text{ for } x \in B' \\
0 \qquad & \text{ otherwise},
\end{cases}
\end{align*}
and let $\nu(x) = \nabla d(x)$. Then, consider
\begin{equation*}
v(x) = f(\nu(x))d(x) + p\epsilon d(x)^{1-s},
\end{equation*}
which is a subsolution in a neighborhood of $\partial B'$ per Lemma \ref{Lemma Test Functions}.
Here, we note that
\begin{equation*}
d(x) = (x_n - \epsilon A|x'|^2)_{+} + o(\epsilon^2).
\end{equation*}
Also, we calculate that
\begin{equation*}
\frac{f(\nu(x))d(x) - x_n}{\epsilon} = -A|x'|^2 - 2Ax_n(\nabla f(e_n) \cdot x') + o(\epsilon).
\end{equation*}
Additionally, note that $\frac{\omega'}{\omega_n} = -\nabla f(e_n)$, which can be seen by differentiating
\begin{equation*}
h(f(\nu)\nu) = 0
\end{equation*}
at $\nu = e_n$. It then follows that
\begin{align*}
v_{\epsilon} := \frac{v(x) - x_n}{\epsilon} = -A|x'|^2 - 2Ax_n(\nabla f(e_n) \cdot x') + px_n^{1-s} + o(\epsilon) \\ =
-A|x' - \frac{\omega}{\omega_n}x_n|^2 - |\frac{\omega}{\omega_n}|^2x_n^2 + px_n^{1-s} + o(\epsilon)
\end{align*}
converges to a function that touches $\phi$, and thus $w$, by below at the origin. Then, the function $v_{\epsilon}(x - Ce_n)$ will for some values of $C$ and $\epsilon$ touch some $w_j$ by below at some point close to the origin, thus yielding a contradiction.

Finally, we may apply the exact same reasoning of Theorem 7.4 in \cite{de2022alt} to reach a contradiction. For completeness, we sketch out the main steps. 

Since it is now known that $\bar{w}$ solves the linearized problem \eqref{eq linearized}, we may show the $C^{1,\alpha}$ estimate
\begin{equation}\label{eq C^{1 alpha} linear}
|\bar{w} - a \cdot x| \leq C|x|^{1 + \alpha}   
\end{equation}
for a universal $\alpha > 0$. By the convergence of the $\bar{w}_j$'s, we can conclude that for $r \leq c$
\begin{equation*}
|\bar{w}_j - a \cdot x| \leq Cr^{1 + \alpha} \quad \text{in $B_r$},  
\end{equation*}
for some fixed vector $a$ such that $a \cdot \omega = 0$. Equivalently, we have that
\begin{equation}
x_n + \epsilon_j(a \cdot x) -\epsilon_{j}\frac{r}{4} \leq w_j \leq x_n + \epsilon_j (a \cdot x) + \epsilon_j \frac{r}{4} \quad \text{ in } B_r.
\end{equation}
Then, letting $\nu_j = \frac{e_n + \epsilon_j a}{\sqrt{\epsilon_j^2|a'|^2 + (a_n + 1)^2}}$, we may use $a \cdot \omega = 0$ to conclude that
\begin{equation}
x_n + \epsilon_j(a \cdot x) - \frac{\epsilon_j r}{4} \leq f_j(\nu_j)(x \cdot \nu_j) \leq x_n + \epsilon_j(a \cdot x) + \frac{\epsilon_j r}{4} \quad \text{ in } B_r.   
\end{equation}
Hence, we conclude that
\begin{equation}
f_j(\nu_j)(x \cdot \nu_j) - \frac{\epsilon_jr}{2} \leq w_j \leq (f_j(\nu_j)(x \cdot \nu_j) + \frac{\epsilon_jr}{2})_{+} \quad \text{ in } B_r, 
\end{equation}
and we have reached a contradiction.
\end{proof}

It then follows from standard arguments that the free boundary is a $C^{1,\alpha}$ graph near the origin for some $\alpha > 0$. However, we have only shown that the free boundary is smooth near points that satisfy a flatness assumption, and we cannot quantify the number of flat points in general. Yet, in the next proposition, we will show that if the free boundary admits a tangent ball from the zero set at a free boundary point $x_0$, then the free boundary is flat at $x_0$.
\begin{proposition}\label{Proposition Exterior Ball}
Let $w$ be a solution to \eqref{eq main problem} with $0 \in FB(w)$. Further, suppose that $e_n \in \partial D$ and that there exists some $r > 0$ such that $B_r(-re_n) \subset \{w = 0\}$. 
Then,  we have that for all $\epsilon$ sufficiently small,
\begin{equation}\label{eq flatness}
(x_n - \epsilon)_{+} \leq w(x) \leq (x_n + \epsilon)_{+} \qquad \text{ in } B_{2\epsilon}.
\end{equation}
\end{proposition}
\begin{proof}
For the estimate from above, we simply recall that 
\begin{equation*}
w(x) \leq Cdist(x, FB(w)) \leq Cdist(x, B_{r}(-re_n))
\end{equation*}
by the optimal regularity estimate in Proposition  \ref{prop optimal growth}. Then, it may be checked by a direct computation that
\begin{equation*}
Cdist(x, B_{r}(-re_n)) \leq (x_n + \epsilon)_{+} \qquad \text{ in $B_{2\epsilon}$}
\end{equation*}
for $\epsilon$ sufficiently small.

As for the estimate from below, suppose such estimate is not true by contradiction. Then, for any fixed $\epsilon > 0$,  there exists a sequence of points $\{x_j\}$ such that $w(x_j) \leq (g(x_j) - \epsilon)_{+}$, where $g(x) = x_n$. We may then apply the reasoning found in the proof of the Semi-Harnack Inequality Lemma \ref{Lemma SemiHarnack} to conclude that in fact 
\begin{equation*}
w(x) \leq (x_n + \epsilon - \delta),
\end{equation*}
where $\delta \leq c$ is universal and independent of $\epsilon$. Taking $\epsilon$ small enough then contradicts $0 \in FB(w)$.
\end{proof}

\section{Connection with minimal surfaces}\label{sec minimal surface}
The purpose of this section is to present the proof of Theorem \ref{thm minimal surfaces}. First, we show that any function that solves the interior equation \eqref{eq interior} in its positivity set already satisfies a growth estimate near the free boundary. Note that this applies to functions that do not satisfy a free boundary condition; this is roughly the multi-dimensional analogue of the fact that the first term in the expansion of a one-dimensional solution must be linear. We will use this estimate when treating the limiting equation in Theorem \ref{thm minimal surfaces}.
\begin{lemma}\label{Lemma Growth Free Boundary}
Let $w \geq 0$ be a solution to \eqref{eq interior} in $B_{1}^{+}$. Furthermore, suppose that $e_n \in \partial D$, that $0 \in \partial\{w >0\}$, and that there is some $r > 0$ and a ball $B_{r}(re_n) \subset \{w > 0\}$ such that $FB(w)$ is tangent to $B_{r}(re_n)$. Then, it follows that near the origin,
\begin{equation}\label{eq flatness}
w(x) \geq x_n + o(|x_n|). 
\end{equation}
\end{lemma}
\begin{proof}
First, consider the solution $\bar{w}$ of the interior equation \eqref{eq interior} in $B_{r}(re_n)$ with zero boundary data. It must then be the case that $w \geq \bar{w}$; notice that the solutions of the interior equation \eqref{eq interior} in $B_{r - t}(e_n)$ with zero boundary data form a continuous family of solutions that would otherwise touch $w$ by below at an interior point. Now, we will show that 
\begin{equation*}
\bar{w}(x) = x_n + o(|x_n|)    
\end{equation*}
near the origin. By a Lipschitz rescaling, we may assume that $r = 1$. First, notice that since $e_n \in \partial D$, we have that the functions
\begin{equation*}
\phi_t(x) = (1 + t)x_n
\end{equation*}
are a continuous family of supersolutions to the interior equation \eqref{eq interior}. If $\bar{w} - x_n$ achieves a negative minimum in $B_1$, then some $\phi_t$ will necessarily touch some $\bar{w}$ by below at an interior point. Hence, we can conclude that 
\begin{equation*}
\bar{w} \leq x_n.
\end{equation*}
Additionally, notice that we have that the family of ``bump" subsolutions
\begin{equation*}
\psi_s = s(1 - |x - e_n|^2)
\end{equation*}
for $s$ sufficiently small universal. As before, we may conclude that $\bar{w} \geq \psi_s$ for any $\psi_s$ whose $s$ parameter is small enough so that $\psi_c$ is a subsolution to the interior equation \eqref{eq interior}. This then implies that
\begin{equation*}
\bar{w} \geq sx_n.
\end{equation*}
Now, assume by way of contradiction that there exists a sequence of points $\{x_j\}$ converging to the origin such that,
\begin{equation*}
\frac{\bar{w}(x_j)}{g(x_j)} \rightarrow s_0 < 1,
\end{equation*}
where $g(x) = x_n$. Then, for $r_j = g(x_j)$, consider the rescalings
\begin{equation*}
\bar{w}_j(x) = \frac{\bar{w}(r_j(x-e_n) + x_j)}{r_j}.
\end{equation*}
By the regularity estimates on $\bar{w}$, such sequence of functions will converge to a limiting function $\bar{w}_0$ that is a global solution of \eqref{eq interior} in $\{x_n > 0\}$. Then, since 

\begin{equation*}
\bar{w}_j(e_n) = \frac{\bar w(x_j)}{r_j} \rightarrow c < 1,    
\end{equation*} 
it follows that $\bar{w}_0(e_n) = c$, and thus the subsolution $cx_n$ touches $\bar{w}_0$ by below at $e_n$, a contradiction.
\end{proof}

Before we do so, we recall the definition of what it means to be a minimal surface in the viscosity sense.

\begin{definition}
Let $w$ be a non-negative continuous function in a bounded, open set $\Omega$. We say that a free boundary $\partial\{w > 0\}$ is a minimal surface in the viscosity sense if the following two conditions hold:
\begin{itemize}
    \item No smooth hypersurface $\Gamma$ with positive mean curvature that lies in the positivity set of $w$ may touch $\partial\{w > 0\}$ tangentially.
    \item No smooth hypersurface $\Gamma$ with negative mean curvature that lies in the zero-set of $w$ may touch $\partial\{w > 0\}$ tangentially.    
\end{itemize}
\end{definition}
\begin{remark}\label{Remark Minimal Surfaces Test Functions}
The idea and motivation behind Theorem \ref{thm minimal surfaces} can be traced back to Lemma \ref{Lemma Test Functions}. In such Lemma, we show that test functions of the form 
\begin{equation*}
f(\nu(x))d(x) \pm \epsilon d(x)^{1 + \mu}
\end{equation*}
are subsolutions ($+$) or supersolutions ($-$) of the interior equation \eqref{eq interior} in a neighborhood of the free boundary as long as
\begin{equation*}
\mu > \sup_{\nu \in \mathbb{S}^{n-1}} \frac{\nabla h(f(\nu)\nu) \cdot \nu}{f(\nu)} \quad \text{ and } 1 > \sup_{\nu \in \mathbb{S}^{n-1}} \frac{\nabla h(f(\nu)\nu) \cdot \nu}{f(\nu)}.
\end{equation*}
However, in the limiting case of Theorem \ref{thm minimal surfaces},
\begin{equation*}
\frac{\nabla h(f(\nu)\nu) \cdot \nu}{f(\nu)} \equiv 1.
\end{equation*}
In this case, the mean curvature of the free boundary becomes relevant. Namely, the computation leading to Lemma \ref{Lemma Test Functions} now shows that the test functions
\begin{equation*}
f(\nu(x))d(x) - Md^2(x)
\end{equation*}
are subsolutions of the interior equation \eqref{eq interior} in a tubular neighborhood of width $\frac{c}{M}$ of the free boundary in the parts of the free boundary with positive mean curvature (namely, where $\Delta d > 0$). Likewise, the test functions
\begin{equation*}
f(\nu(x))d(x) + Md^2(x)
\end{equation*}
are supersolutions of the interior equation \eqref{eq interior} in a tubular neighborhood of width $\frac{c}{M}$ of the free boundary in the parts of the free boundary with negative mean curvature (where $\Delta d < 0$). These test functions can then be modified to touch one of the approximating equations, which then implies that the free boundary must be a minimal surface. We develop this notion in the proof below. 
\end{remark}
\begin{proof}[Proof of Theorem \ref{thm minimal surfaces}]
The fact that $w$ satisfies the interior equation \eqref{eq min surf h} in the viscosity sense follows trivially from the stability of viscosity solutions under uniform limits. Regarding the free boundary being a minimal surface, let us suppose that $0 \in FB(w)$. Then, suppose by way of contradiction that there exists an oriented smooth hypersurface $\Gamma$ lying in $\{w > 0\}$ that touches $FB(w)$ tangentially at the origin and such that $\kappa(0) > 0$, where $\kappa$ is the mean curvature of $\Gamma$ at the origin. For $r$ sufficiently small, $\Gamma$ will split $B_{r}$ into two regions; let $A$ denote the region such that $A \subset \{w > 0\}$. Then, let 

\begin{align*}
d(x) = 
\begin{cases}
dist(x, \Gamma) \qquad & \text{ for } x \in A \\
0 \qquad & \text{ otherwise},
\end{cases}
\end{align*}
and let $\nu(x) = \nabla d(x)$. Consider the function
\begin{equation}\label{eq test 1}
v(x) = f(\nu(x))d(x) - Md^2(x).
\end{equation}
By the computation in subsection \ref{sub free boundary}, we can conclude that $v$ is a subsolution to the interior equation \eqref{eq main problem} in the set $\{d \leq \frac{\delta}{M}\}$ for some $\delta \leq c$. On $\{d = \frac{\delta}{M}\}$, we have that
\begin{equation*}
v(x) = \delta[\frac{f(\nu(x)) - \delta}{M}].
\end{equation*}
Using the Growth Lemma \ref{Lemma Growth Free Boundary}, we also have that
\begin{equation*}
w(x) = f(\nu(x))d(x) + o(d(x)) \geq (f(\nu(x) - \delta)d(x) \quad \text{ on $\Omega_{\delta}$},
\end{equation*}
where $\Omega_{\delta}$ is a neighborhood of the origin whose size depends on $\delta$. Setting $M$ large enough then ensures that $\{d \leq \frac{\delta}{M}\} \subset \Omega_{\delta}$, and so we can conclude that $v < w$ in $\{d = \frac{\delta}{M}\}$. Then, if $w - v$ attains a negative minimum in $\{d \leq \frac{\delta}{M}\}$, then a function of the form $v(x - t\eta)$ will touch $w$ by below for some $t > 0$ if we let $\eta$ be the unit normal to $\Gamma$ at the origin that points towards the positivity set of $w$. Thus, $v$ touches $w$ by below at the origin. 

Now, for $\epsilon > 0$ small and any $\mu \in (0,1)$ to be fixed later, consider 
\begin{equation}
u(x) = f(\nu(x))d(x) + \epsilon d(x)^{1 + \mu} - Md^2(x).
\end{equation}
\noindent
For $\epsilon$ sufficiently small and $\mu$ sufficiently close to $1$, there exists some neighborhood $\mathcal{N}$ of the origin and some $w_j$ in the approximating sequence of solutions such that $u < w_j$ on $\partial \mathcal{N} \setminus \Gamma$ and such that $u$ is a subsolution to the interior equation \eqref{eq interior} corresponding to $w_j$. If $0$ is an interior point of $\{w_j = 0\}$, then the translated function $u(x - t\eta)$ will for some choice of $t$ touch $w_j$ by below either at an interior or a free boundary point (where $\eta$ again is the unit normal of $\Gamma$ pointing towards $\{d > 0\}$), providing a contradiction. So, we are left with the possibility that $0 \in FB(w_j)$ or $w_j(0) > 0$. In the former case, either $u$ touches $w_j$ by below at the origin or the same translation as before will touch $w_j$ at an interior or free boundary point. Hence, the only remaining possibility is that $w_j(0) > 0$, so that upon possibly shrinking $\mathcal{N}$ we have that $w_j > w$ on $\{d \leq M\}$. We may then assume without loss of generality that $w_{i} \geq w$ for all $i \in \mathbb{N}$. Since $FB(w_i) \rightarrow FB(w)$ in Hausdorff distance, we can then conclude that a translation of the form $u(x + t\eta)$ will for an appropriate choice of $t$ touch some $w_j$ by below at a free boundary point, yielding a contradiction.  

Using the same reasoning as above, we can show that no smooth oriented hypersurface of mean negative curvature can touch the free boundary from the zero set of $w$, thus proving the result.
\end{proof}

\begin{remark}
The previous theorem already assumes the convergence of $\{w_k\}, \{h_k\}, \partial\{w_k > 0\}$. In the case of the Alt-Phillips equation for negative power potentials, the convergence is guaranteed via variational methods for critical points. Additionally, if the sequence of solutions all have the same right-hand side $h$ and domain function $f$, the convergence can also be guaranteed by Remark \ref{rmrk convergence for fixed right-hand side}. In a subsequent work, we will explore conditions under which uniform convergence of viscosity solutions can be guaranteed in more general settings.

\end{remark}
\bibliographystyle{acm}
\bibliography{References}
\end{document}